\documentclass[12pt]{article}
\usepackage{amsopn,amssymb,amsthm,amsmath,bm}
\usepackage{paralist}
\usepackage{algorithm,algorithmic}
\usepackage{color}
\usepackage{graphicx}
\usepackage{comment}
\usepackage{hyperref}
\hypersetup{
     colorlinks   = true,
     linkcolor    = blue,
     citecolor    = green
}
\usepackage{multirow}
\usepackage{float}
\usepackage{subfigure}
\usepackage{enumitem}
\usepackage[style=numeric,natbib=true,backend=bibtex, sortcites=true]{biblatex}

\usepackage{thmtools}
\usepackage{thm-restate}
\usepackage{cleveref}

\bibliography{references}
\usepackage{framed}
\newlength{\defbaselineskip}
\usepackage[a4paper , total={8.5in, 11in}, margin=1in]{geometry}

\usepackage{enumitem}
\setlist[enumerate,1]{leftmargin=*,wide=0em, label = {\bfseries \roman*.}}
\setlist[itemize,1]{leftmargin=*,wide=0em}
\newcommand\bbR{\ensuremath{\mathbb{R}}} % Real numbers
\newcommand\bbE{\ensuremath{\mathbb{E}}} % Expectation
\DeclareMathOperator*{\diag}{diag} % Diagonal matrix
\newcommand{\Abs}[1]{\left |#1\right|}
\newcommand{\norm}[1]{{\left\|#1\right\|}}

\newcommand{\Prob}[1]{{\Pr}\left(#1\right)}

\newtheorem{theorem}{Theorem}

\newtheorem{definition}{Definition}
\newtheorem{lemma}{Lemma}

\newtheorem{assumption}{Assumption}

\newtheorem{condition}{Condition}

\newcommand*\lin[1]{\langle #1\rangle}
\newcommand {\Ex} { {\mathbb E} }



\newif\ifwithcomments
\withcommentstrue
\ifwithcomments
\newcommand{\xxx}[1]{{\color{red} (#1)}}
\else
\newcommand{\xxx}[1]{}
\fi

\newcommand{\red}[1]{\textcolor{red}{#1}}

\DeclareMathOperator{\bigO}{\mathcal{O}}

\def\lmin{\lambda_{\min}}
\newcommand{\vv}[1] {\mathbf{#1}}

\def\u{\vv{u}}

\def\p{\vv{p}}

\def\s{\vv{s}}

\def\x{\vv{x}}
\def\y{\vv{y}}
\def\w{\vv{w}}

\def\A{\vv{A}}
\def\B{\vv{B}}

\def\a{\vv{a}}

\def\X{\vv{X}}
\def\Y{\vv{Y}}
\def\U{\vv{U}}
\def\V{\vv{V}}

\def\S{\vv{S}}

\def\H{\vv{H}}

\begin{document}
\title{\Large Newton-Type Methods for Non-Convex Optimization Under Inexact Hessian Information}
\author{
	Peng Xu
	\thanks{
		Institute for Computational and Mathematical Engineering,
		Stanford University,
		%Stanford, CA 94305.
		Email: pengxu@stanford.edu
	}
	\and
	Fred Roosta
	\thanks{
		School of Mathematics and Physics, University of Queensland, Brisbane, Australia, and 
		International Computer Science Institute, Berkeley, USA,     
		Email: fred.roosta@uq.edu.au
	}
	\and
	Michael W. Mahoney
	\thanks{
	International Computer Science Institute and Department of Statistics, University of California at Berkeley,    
	Email: mmahoney@stat.berkeley.edu
}
}

\date{\today}
\maketitle
	
%%%%%%%%%%%%%%%%%%%%%%%%%%%%%%%%%%%%%%%%%%%%%%%
% 				 	Abstract   	 					          %	
%%%%%%%%%%%%%%%%%%%%%%%%%%%%%%%%%%%%%%%%%%%%%%%
\begin{abstract}
We consider variants of trust-region and adaptive cubic regularization methods for non-convex optimization, in which the Hessian matrix is approximated. Under certain condition on the inexact Hessian, and using approximate solution of the corresponding sub-problems, we provide iteration complexity to achieve $ \epsilon $-approximate second-order optimality which have been shown to be tight. 
Our Hessian approximation condition offers a range of advantages as compared with the prior works and allows for direct construction of the approximate Hessian with a priori guarantees through various techniques, including randomized sampling methods. In this light, we consider the canonical problem of finite-sum minimization, provide appropriate uniform and non-uniform sub-sampling strategies to construct such Hessian approximations, and obtain optimal iteration complexity for the corresponding sub-sampled trust-region and adaptive cubic regularization methods. 
\end{abstract}

% -----------------------------------------------------------------------
%     background
% -----------------------------------------------------------------------
\section{Introduction}
\label{sec:intro}
Consider the generic unconstrained optimization problem 
\begin{align}
\label{eq:obj}
\min_{\x \in \bbR^d} F(\x), \tag{{\bf P0}}
\end{align}
where $ F: \mathbb{R}^{d} \rightarrow \mathbb{R} $ is \emph{smooth} and \emph{non-convex}.  
%Over the last few decades, many optimization algorithms have been developed to solve \eqref{eq:obj}, e.g., \cite{bertsekas1999nonlinear,nocedal2006numerical}. 
Faced with the large-scale nature of modern ``big-data'' problems, many of the classical optimization algorithms might prove to be inefficient, if applicable at all. In this light, many of the recent research efforts have been centered around designing variants of classical algorithms which, by employing suitable \emph{approximations} of the gradient and/or Hessian, improve upon the cost-per-iteration, while maintaining the original iteration complexity. 
%Indeed, designing variants of classical methods which can strike a balance between per-iteration costs and iteration complexity has been at center-stage in fueling the research in optimization for machine learning and data analysis applications. 
In this light, we focus on trust-region (TR) \cite{conn2000trust} and cubic regularization (CR) \cite{cubic1981Griewank}, two algorithms which are considered as among the most elegant and theoretically sound general-purpose Newton-type methods for non-convex problems. 

In doing so, we first consider \eqref{eq:obj}, and study the theoretical convergence properties of variants of these two algorithms in which, under favorable conditions, Hessian is suitably approximated.
We show that our Hessian approximation conditions, in many cases, are weaker than the existing ones in the literature. In addition, and in contrast to some prior works, our conditions allow for efficient constructions of the inexact Hessian with a priori guarantees via various approximation methods, of which Randomized Numerical Linear Algebra (RandNLA), \cite{mahoney2011randomized,drineas2016randnla}, techniques are shown to be highly effective.
%Such relaxed condition allows for efficient constructions of the inexact Hessian via various simple approximation methods, of which randomized techniques are shown to be particularly effective. 

Subsequently, to showcase the application of randomized techniques for construction of the approximate Hessian, we consider an important instance of \eqref{eq:obj}, i.e., large-scale \emph{finite-sum} minimization, of the form
\begin{align}
\label{eq:obj_sum}
\min_{\x \in \bbR^d} F(\x) \triangleq \frac{1}{n}\sum_{i=1}^n f_i(\x), \tag{{\bf P1}}
\end{align} 
and its special case
\begin{align}
\label{eq:obj_sum_ERM}
\min_{\x\in\bbR^d} F(\x)  \triangleq \frac{1}{n} \sum_{i=1}^n f_i(\a_i^T\x), \tag{{\bf P2}}
\end{align}
where $ n \gg 1 $, each $ f_{i} $ is a smooth but possibly non-convex function, and $\a_i\in\bbR^d, i = 1,\ldots, n,$ are given.
Problems of the form \eqref{eq:obj_sum} and \eqref{eq:obj_sum_ERM} arise very often in machine learning, e.g., \cite{shalev2014understanding} as well as scientific computing, e.g., \cite{rodoas1,rodoas2}. In big-data regime where $ n \gg 1 $, operations with the Hessian of $ F $, e.g., matrix-vector products, typically constitute the main bottleneck of computations. Here, we show that our relaxed Hessian approximation conditions allow one to draw upon the \emph{sub-sampling} ideas of \cite{roosta2018sub,xu2016sub,bollapragada2016exact}, to design variants of TR and CR algorithms where the Hessian is  \emph{(non-)uniformly} sub-sampled. We then present the theoretical convergence properties of these variants for non-convex finite-sum problems of the form \eqref{eq:obj_sum} and \eqref{eq:obj_sum_ERM}. 

%This paper is mainly motivated by developing novel theory which supports simple constructions of approximate Hessian in practice. Extensive numerical examples demonstrating various practical aspects of our proposed algorithms for solving \eqref{eq:obj_sum} and \eqref{eq:obj_sum_ERM} are, instead, given in \cite{xuNonconvexEmpirical2017}. Specifically, in \cite{xuNonconvexEmpirical2017}, we consider two classes of non-convex optimization problems that arise often in practice, i.e. \emph{non-linear least squares} as well as \emph{deep learning} and present extensive numerical experiments illustrating the empirical performance of the sub-sampled methods considered in this paper on both, \emph{real} and \emph{synthetic} data.

The rest of this paper is organized as follows. 
In Section \ref{sec:notation}, we first introduce the notation and definitions used throughout the paper.
For completeness, in Section \ref{sec:background_related_work}, we give a brief review of trust region (Section \ref{sec:TR}) and cubic regularization (Section \ref{sec:CR}) along with related prior works.
%we give a brief review of the classical Newton-type methods considered in this paper, namely trust region in Section \ref{sec:TR} and cubic regularization in Section \ref{sec:CR}.  
%In the process, we briefly mention the works related to the present paper and in its light, outline, in details, 
Our main contributions are summarized in Section \ref{sec:contributions}.  
Theoretical analysis of the proposed algorithms for solving generic non-convex problem \eqref{eq:obj} are presented in Section \ref{sec:convergence_analysis}. Various randomized sub-sampling strategies as well as theoretical properties of the proposed algorithms for finite-sum minimization problems \eqref{eq:obj_sum} and \eqref{eq:obj_sum_ERM} are given in Section \ref{sec:sub_sampling_and_finite_sum}.
%, we consider finite-sum minimization problems \eqref{eq:obj_sum} and \eqref{eq:obj_sum_ERM}, introduce various randomized sub-sampling strategies as a way to construct inexact Hessian information, and give the corresponding algorithmic convergence guarantees. 
Conclusions and further thoughts are gathered in Section \ref{sec:conclusions}. 

\subsection{Notation and Definitions}
\label{sec:notation}

Throughout the paper, vectors are denoted by bold lowercase letters, e.g., $\vv{v}$, and matrices or random variables are denoted by bold upper case letters, e.g., $\V$. 
$ \vv{v}^{T} $ denotes the transpose of a real vector $\vv{v}$. We use regular lower-case and upper-case letters to denote scalar constants, e.g., $ c $  or $ K $. For two vectors, $ \vv{v},\w $, their inner-product is denoted as $ \lin{\vv{v}, \w}  = \vv{v}^{T} \w$. For a vector $\vv{v}$, and a matrix $\V$, $\|\vv{v}\|$ and $\|\V\|$ denote the vector $\ell_{2}$ norm and the matrix spectral norm, respectively, while $\|\V\|_{F}$ is the matrix Frobenius norm. $\nabla F(\x)$ and $\nabla^{2} F(\x)$ are the gradient and the Hessian of $F$ at $\x$, respectively, and $\mathbb{I}$ denotes the identity matrix. For two symmetric matrices $\A$ and $\B$, $\A \succeq \B$ indicates that $\A-\B$ is symmetric positive semi-definite. The subscript, e.g., $\x_{t}$, denotes iteration counter and $\log(x)$ is the natural logarithm of $x$. The inexact Hessian is denoted by $ \H(\x) $, but for notational simplicity, we may use $ \H_{t} $ to, instead, denote the approximate Hessian evaluated at the iterate $ \x_{t} $ in iteration $ t $, i.e., $\H_{t} \triangleq  \H(\x_{t}) $. Throughout the paper, $\mathcal{S}$ denotes a collection of indices from $\{1,2,\cdots,n\}$, with potentially repeated items and its cardinality is denoted by $|\mathcal{S}|$.

%A very useful property of convex functions is that ``local optimality'' and ``global optimality'' are in fact, the same. Unfortunately, in non-convex settings, this is no longer the case. For example, even optimization of a degree four polynomial can be NP-hard \cite{hillar2013most}. In fact, just checking whether a point is not a local minimum is NP-complete \cite{murty1987some}. Thus, 
Unlike convex functions for which ``local optimality'' and ``global optimality'' are in fact the same, in non-convex settings, we are often left with designing algorithms that can guarantee convergence to approximate local optimality. In this light, throughout this paper, we make use of the following definition of $(\epsilon_g, \epsilon_H)$-Optimality:
\begin{definition}[$(\epsilon_g, \epsilon_H)$-Optimality]
	\label{def:opt} 
	Given $\epsilon_g, \epsilon_H\in(0,1)$, $\x\in\bbR^d$ is an $(\epsilon_g, \epsilon_H)$-optimal solution to the problem \eqref{eq:obj}, if 
	\begin{align}
		\label{eq:2ndopt}
		\|\nabla F(\x)\| \le \epsilon_g, ~~ \lambda_{\min} (\nabla^2 F(\x)) \ge -\epsilon_H.
	\end{align}
\end{definition}
We note that $(\epsilon_g, \epsilon_H)$-Optimality (even with $\epsilon_g = \epsilon_H = 0$) does not necessarily imply closeness to any local minimum, neither in iterate nor in the objective value. However, if the saddle points satisfy the strict-saddle property \cite{ge2015escaping,lee2016gradient}, then an $(\epsilon_g, \epsilon_H)$-optimality guarantees vicinity to a local minimum for sufficiently small $\epsilon_g$ and $\epsilon_H$.

\subsection{Background and Related Work}
\label{sec:background_related_work}

Arguably, the most straightforward approach for \emph{globalization} of many Newton-type algorithms is the application of line-search. However, near saddle points where the gradient magnitude can be small, traditional line search methods can be very ineffective and in fact produce iterates that can get stuck at a saddle point \cite{nocedal2006numerical}. Trust region and cubic regularization methods are two elegant globalization alternatives that, specially recently, have attracted much attention. %In this section, we give a brief review the classical TR and CR algorithms for solving \eqref{eq:obj}. 
%There are many similarities between TR and CR algorithms in terms of their theoretical and algorithmic properties as well as methods for solving their respective sub-problems. 
The main advantage  of these methods is that they are reliably able to take advantage of the direction of negative curvature and escape saddle points. 
%More specifically, if the Hessian at a saddle point contains a negative eigenvalue, these methods can leverage the corresponding direction of negative curvature to obtain decrease in the objective function values. Many problems of interest exhibit saddle points which include negative curvature direction \cite{ge2015escaping,sun2015nonconvex}. 
%As a result, TR and CR methods, if made scalable, can be very effective for solving many large-scale non-convex problems of interest for machine learning and scientific computing. 
In this section we briefly review these algorithms as they pertain to the present paper and mention the relevant prior works. 

\subsubsection{Trust Region}
\label{sec:TR}

TR methods \cite{sorensen1982newton,conn2000trust} encompass a general class of iterative methods which specifically define a region around the current iterate within which they trust the model to be a reasonable approximation of 
the true objective function. 
%They then find the step as a (approximate) minimizer of the model in this region. In effect, they choose the direction and length of the step simultaneously. If a step is not acceptable, they reduce the size of the region and find a new minimizer. 
The most widely used approximating model, which we consider here, is done via a quadratic function.
%obtained from the second-order Taylor expansion of the true objective at the current iterate. 
More specifically, using the current iterate $\x_{t}$, the quadratic variant of TR algorithm finds the next iterate as $\x_{t+1} = \x_{t} + \s_{t}$ where $\s_{t}$ is a solution of the \emph{constrained} sub-problem
\begin{subequations}
\begin{align}
	\label{eq:TR_sub}
	\min \; ~~& m_t(\s) \triangleq \lin{\s, \nabla F(\x_{t})} + \frac{1}{2} \lin{ \s, \nabla^{2} F(\x_{t})\s}  \\
	\text{s.t. } ~~& \|\s\|_{2} \leq \Delta_{t} \nonumber.
\end{align}
Here, $ \Delta_{t} $ is the region in which we ``trust'' our quadratic model to be an acceptable approximation of the true objective for the current iteration. The major bottleneck of computations in TR algorithm is the minimization of the constrained quadratic sub-problem \eqref{eq:TR_sub}, for which numerous approaches have been proposed, e.g., \cite{steihaug1983conjugate,more1983computing,sorensen1997minimization,gould1999solving,lenders2016trlib,erway2009iterative,gould2010solving,hazan2015linear}.

For a smooth non-convex objective and in order to obtain approximate first-order criticality, i.e., $ \|\nabla F(\x_{t})\| \leq \epsilon_{g}$ for some $\epsilon_{g} \in (0,1)$, the complexity of an (inexact) trust-region method, which ensures at least a Cauchy (steepest-descent-like) decrease at each iteration, is shown to be of the same order as that of steepest descent, i.e., $ \mathcal{O}(\epsilon_{g}^{-2}) $; e.g., \cite{gratton2008recursiveI,gratton2008recursiveII,blanchet2016convergence,gratton2017complexity,cartis2012complexity}. Recent non-trivial modifications of the classical TR methods have also been proposed which improve upon the complexity to $ \mathcal{O}(\epsilon_{g}^{-3/2}) $; see \cite{curtis2017trust} and further extensions to a more general framework in \cite{curtis2017inexact}.
These bounds can be shown to be tight \cite{cartis2010complexity} in the worst case. Under a more general algorithmic framework and in terms of objective function sub-optimality, i.e., $ F(\x) - F^{*} \leq \epsilon $, better complexity bounds, in the convex and strongly-convex settings, have been obtained which are of the orders of $ \mathcal{O}(\epsilon_{g}^{-1}) $ and $ \mathcal{O}(\log ({1}/{\epsilon_{g}})) $, respectively \cite{grapiglia2016worst}. 

For non-convex problems, however, it is more desired to obtain complexity bounds for achieving approximate second-order criticality, i.e., Definition \ref{def:opt}. For this, bounds in the orders of $ \mathcal{O}(\max\{\epsilon_{H}^{-1}\epsilon_{g}^{-2}, \epsilon_{H}^{-3}\}) $ and $ \mathcal{O}(\max\{\epsilon_{g}^{-3}, \epsilon_{H}^{-3}\}) $ have been obtained in \cite{cartis2012complexity} and \cite{grapiglia2016worst}, respectively. Similar bounds were also given in \cite{gratton2017complexity} under probabilistic model. %Most these complexities are obtained for trust region-type methods which, in addition to Cauchy decrease, guarantee a descent, at least, as good as that obtained from following the negative curvature direction (if present). 
Bounds of this order have shown to be optimal in certain cases \cite{cartis2012complexity}.

More closely related to the present paper, there have been several results which study the role of derivative-free and probabilistic models in general, and Hessian approximation in particular, e.g., see \cite{cartis2012complexity,conn2009global,chen2015stochastic,blanchet2016convergence,bandeira2014convergence,larson2016stochastic,shashaani2016astro,gratton2017complexity} and references therein.

\subsubsection{Cubic Regularization}
\label{sec:CR}
An alternative to the traditional line-search and TR for globalization of Newton-type methods is the application of cubic regularization. Such class of methods is characterized by generating iterates as $\x_{t+1} = \x_{t} + \s_{t}$ where $\s_{t}$ is a solution of the following \emph{unconstrained} sub-problem 
\begin{align}
	\label{eq:CR_sub}
	\min_{\s \in \mathbb{R}^{d}} \; ~~& m_t(\s) \triangleq \lin{\s, \nabla F(\x_{t})} + \frac{1}{2} \lin{ \s, \nabla^{2} F(\x_{t})\s} + \frac{\sigma_{t}}{3} \|\s\|^{3},
\end{align}
\end{subequations}
where $ \sigma_{t} $ is the cubic regularization parameter chosen for the current iteration. 
%Naively speaking, one can view $ \sigma_{t} $ as the reciprocal of $ \Delta_{t} $. 
As in the case of TR, the major bottleneck of CR involved solving the sub-problem \eqref{eq:CR_sub}, for which various techniques have been proposed, e.g., \cite{cartis2011adaptiveI,carmon2016gradient,bianconcini2015use,agarwal2016finding}.
%the role of the parameter $ \sigma_{t} $ is very similar to the trust-region radius, $ \Delta_{t} $. In fact, one can view $ \sigma_{t} $ as the reciprocal of $ \Delta_{t} $ \cite{cartis2011adaptiveI}. 
%A direct comparison of \cite[Lemmas 3.2 and 3.3]{cartis2011adaptiveII} with \cite[Theorems 6.4.2, 6.4.3]{conn2000trust} also reveals the striking resemblance of the role of the trust region radius, $ \Delta_{t} $, with cubic regularization parameter, $ \sigma_{t} $. 

To the best of our knowledge, the use of such regularization, was first introduced in the pioneering work of \cite{cubic1981Griewank}, and subsequently further studied in the seminal works of \cite{nesterov2006cubic,cartis2011adaptiveI,cartis2011adaptiveII}.%, which provide an in-depth analysis of such regularization methods in a variety of setting. 
From the worst-case complexity point of view, CR has a better dependence on $ \epsilon_{g} $ compared to TR. More specifically, \cite{nesterov2006cubic} showed that, under global Lipschitz continuity assumption on the Hessian, if the sub-problem \eqref{eq:CR_sub} is solved exactly, then the resulting CR algorithm achieves the approximate first-order criticality with complexity of $ \mathcal{O}(\epsilon_{g}^{-3/2})$. These results were extended by the pioneering and seminal works of \cite{cartis2011adaptiveI,cartis2011adaptiveII} to an adaptive variant, which is often referred to as ARC (Adaptive Regularization with Cubics). In particular, the authors showed that the worst case complexity of $\mathcal{O}(\epsilon_{g}^{- 3 / 2})$ can be achieved without requiring the knowledge of the Hessian's Lipschitz constant, access to the exact Hessian, or multi-dimensional global optimization of the sub-problem \eqref{eq:CR_sub}. These results were further refined in \cite{cartis2012complexity} where it was shown that, not only, multi-dimensional global minimization of \eqref{eq:CR_sub} is unnecessary, but also the same complexity can be achieved with mere one or two dimensional search. This $ \mathcal{O}(\epsilon^{-3/2})$ bound has been shown to be tight \cite{cartis2011optimal}. As for the approximate second-order criticality, \cite{cartis2012complexity} showed that at least $ \mathcal{O}(\max\{\epsilon_{g}^{-2}, \epsilon_{H}^{-3}\}) $ is required. With further assumptions on the inexactness of sub-problem solution, \cite{cartis2011adaptiveII,cartis2012complexity} also show that one can achieve $ \mathcal{O}(\max\{\epsilon_{g}^{-3/2}, \epsilon_{H}^{-3}\}) $, which is shown to be tight \cite{cartis2010complexity}. Better dependence on $ \epsilon_{g} $ can be obtained if one assumes additional structure, such as convexity, e.g., see \cite{nesterov2006cubic,cartis2012evaluation} as well as the acceleration scheme of \cite{nesterov2008accelerating}. 

Recently, for (strongly) convex  problems, \cite{ghadimi2017second} obtained sub-optimal complexity for ARC and its accelerated variants using Hessian approximations.
In the context of stochastic optimization problems, \cite{tripuraneni2017stochastic} considers cubic regularization with a priori chosen fixed regularization parameter using both approximations of the gradients and Hessian.
%, while all of our algorithms here are adaptive and practically more appealing. 
Specific to the finite-sum problem \eqref{eq:obj_sum}, and by a direct application of the theoretical results of \cite{cartis2011adaptiveI,cartis2011adaptiveII}, \cite{kohler2017subsampledcubic} presents a sub-sampled variant of ARC, in which the exact Hessian and the gradient are replaced by sub-samples. However, unfortunately, their analysis suffers from a rather vicious circle: the approximate Hessian and gradient are formed based on an \emph{a priori unknown} step which can only be determined after such approximations are formed. 

\subsection{Contributions}
\label{sec:contributions}

In this section, we summarize the key aspects of our contributions. 
In Section \ref{sec:convergence_analysis}, we consider \eqref{eq:obj} and establish the worst-case iteration complexities for variants of trust-region and adaptive cubic regularization methods in which the Hessian is suitably approximated. 
More specifically, our entire analysis is based on the following key condition on the approximate Hessian $\H(\x)$:
\begin{condition}[Inexact Hessian Regularity]
	\label{condition:Hessian_approximation_H} 
	For some $0 < K_{H} < \infty $, $ \epsilon > 0$, the approximating Hessian, $ \H(\x_{t}) $, satisfies
	\begin{subequations}
		\label{eq:Hessian_approximation} 
		\begin{align}
			&\norm{\left(\H(\x_{t}) - \nabla^2 F(\x_{t})\right) \s_{t}}\le \epsilon \cdot\norm{\s_{t}}, \label{eq:Hessian_accuracy_H} \\ % \tag{$ \textbf{H}_{\textbf{a}} $}\\
			&\norm{\H(\x_{t})}\le K_{H}, \label{eq:Hessian_boundedness_H} %\tag{$ \textbf{H}_{\textbf{b}} $}
		\end{align}
	\end{subequations}
	where $ \x_{t} $ and $ \s_{t} $ are, respectively, the iterate and the update at iteration $ t $.
\end{condition}
Under Condition \ref{condition:Hessian_approximation_H}, we show that our proposed algorithms (Algorithms \ref{alg:STR_fg} and \ref{alg:SARC_fg}) achieve \emph{the same worst-case iteration complexity} to obtain \emph{approximate second order critical solution} as that of the exact variants  (Theorems \ref{theorem:STR_main_det}, \ref{theorem:SARC_main_det}, and \ref{theorem:SARC_main_det_optimal}).
%Specifically, not only does \eqref{eq:Hessian_accuracy_H} in many situations impose a milder condition, but also it has an additional advantage that it allows for practical construction of the approximate Hessian with a priori guarantees using a range of techniques. In Section \ref{sec:convergence_analysis}, our emphasis is on obtaining structural convergence results, which can be applied to any Hessian approximation techniques that satisfy Condition \ref{condition:Hessian_approximation_H}, e.g., deterministic approaches using quasi-Newton or finite-difference approximations as well as random constructions such as sketching or sub-sampling.

In Section \ref{sec:sub_sampling_and_finite_sum}, we describe schemes for constructing $ \H(\x_{t}) $ to satisfy Condition \ref{condition:Hessian_approximation_H}. Specifically, in the context of finite-sum optimization framework, i.e., problems \eqref{eq:obj_sum} and \eqref{eq:obj_sum_ERM}, we present various \emph{sub-sampling} schemes to probabilistically ensure Condition \ref{condition:Hessian_approximation_H} (Lemmas \ref{lemma:uniform} and \ref{lemma:nonuniform}). Our proposed randomized sub-sampling strategies guarantee, with high probability, a stronger condition than \eqref{eq:Hessian_accuracy_H}, namely 
\begin{align}
	\label{eq:Hessian_accuracy_H2}
	\|\H(\x) - \nabla^{2} F(\x)\| \leq \epsilon.
\end{align}
It is clear that \eqref{eq:Hessian_accuracy_H2} implies \eqref{eq:Hessian_accuracy_H}.
%The sub-sampling complexities based on \eqref{eq:Hessian_accuracy_H2} do not depend on the step-size, can be set \emph{a priori} and, in turn, be used throughout all iterations.
We then give \emph{optimal} iteration complexities for Algorithms \ref{alg:STR_fg} and \ref{alg:SARC_fg} for optimization of non-convex finite-sum problems where the Hessian is approximated by means of appropriate sub-sampling (Theorems \ref{theorem:STR_main_prob}, \ref{theorem:SARC_main_prob} and \ref{theorem:SARC_main_prob_optimal}).

%To the best of our knowledge, all previous work that considered the inexact Hessian information for trust-region or cubic regularization methods made use of a stronger condition on the approximate Hessian, i.e., 
%\begin{equation}
%\label{eq:Hessian_accuracy_H_quadratic}
%\norm{(\H(\x_t) - \nabla^2 F(\x_t))\s_t } \le C\cdot \norm{\s_t}^2, \tag{{\bf C0}},
%\end{equation}
%for some $C>0$, e.g. \cite{bandeira2014convergence,gratton2017complexity,cartis2011adaptiveI,cartis2011adaptiveII,cartis2012complexity}. Condition \eqref{eq:Hessian_accuracy_H_quadratic} is stronger than the celebrated Dennis-Mor\'{e} \cite{dennis1974characterization}, i.e., 

To establish optimal second-order iteration complexity,  many previous works considered Hessian approximation conditions that, while enjoying many advantages, come with certain disadvantages. Our proposed Condition \ref{condition:Hessian_approximation_H} aims to remedy some of these disadvantages. We first briefly review the conditions used in the prior works, and subsequently highlight the merits of Condition \ref{condition:Hessian_approximation_H} in comparison.

\subsubsection{Conditions Used in Prior Works}
For the analysis of trust-region, many authors have considered the following condition 
\begin{subequations}%{$ \widehat{\textbf{H}}_{\textbf{a}}$}
%\label{eq:Hessian_accuracy_H_quadratic_BOTH}
\begin{align}
\label{eq:Hessian_accuracy_H_quadratic_STR}
\norm{\H(\x_{t}) - \nabla^2 F(\x_{t} + \s)}\le C_{1} \Delta_{t}, \quad \forall \s \in \{\s; \: \norm{\s} \leq \Delta_{t}\}, %\tag{$ \widehat{\textbf{H}}_{\textbf{a}}\textbf{.1} $}
\end{align}
for some $ 0 < C_{1} < \infty $, where $ \Delta_{t} $ is the current trust-region radius, e.g., \cite{bandeira2014convergence,gratton2017complexity}. In \cite{blanchet2016convergence}, condition \eqref{eq:Hessian_accuracy_H_quadratic_STR} is replaced with
\begin{align}
	\label{eq:Hessian_accuracy_H_quadratic_STR_new}
	\norm{\H(\x_{t}) - \nabla^2 F(\x_{t})}\le C_{2} \Delta_{t},
\end{align}
for some $ 0 < C_{2} < \infty $.
%\footnote{The initial two arXiv versions of \cite{blanchet2016convergence} include only first-order complexity results. Condition \eqref{eq:Hessian_accuracy_H_quadratic_STR_new} was introduced in the third version of \cite{blanchet2016convergence}, which added second order complexity analysis and appeared almost a year after the original submission of this work, during the reviewing process.}
In fact, by assuming Lipschitz continuity of Hessian, it is easy to show that \eqref{eq:Hessian_accuracy_H_quadratic_STR} and \eqref{eq:Hessian_accuracy_H_quadratic_STR_new} are equivalent, in that one implies the other, albeit with modified constants.  
We also note that \cite{bandeira2014convergence,gratton2017complexity,blanchet2016convergence} study a more general framework under which the entire sub-problem model is probabilistically constructed and approximation extends beyond just the Hessian.

For cubic regularization, the condition imposed on the inexact Hessian is often considered as 
\begin{align}
\label{eq:Hessian_accuracy_H_quadratic_SARC}
\norm{\left(\H(\x_{t}) - \nabla^2 F(\x_{t})\right) \s_{t}}\le C_{3} \norm{\s_{t}}^{2}, %\tag{$ \widehat{\textbf{H}}_{\textbf{a}}\textbf{.2} $}
\end{align}			
\end{subequations}
for some $ 0 < C_{3} < \infty $, e.g., \cite{cartis2011adaptiveI,cartis2011adaptiveII,cartis2012complexity} and other follow-up works. In fact, \cite{cartis2012complexity} has also established optimal iteration complexity for trust-region algorithm under \eqref{eq:Hessian_accuracy_H_quadratic_SARC}. 
Both of \eqref{eq:Hessian_accuracy_H_quadratic_STR} and \eqref{eq:Hessian_accuracy_H_quadratic_SARC}, are stronger than the celebrated Dennis-Mor\'{e} \cite{dennis1974characterization} condition, i.e., 
\begin{align*}
%\label{eq:Hessian_accuracy_H_Dennis_More}
\lim_{t \rightarrow \infty} \frac{\norm{\left(\H(\x_{t}) - \nabla^2 F(\x_{t})\right) \s_{t}}}{\norm{\s_{t}}} = 0.
\end{align*}
Indeed, under certain assumptions, Dennis-Mor\'{e} condition is satisfied by a number of quasi-Newton methods, although the same cannot be said about \eqref{eq:Hessian_accuracy_H_quadratic_STR} and \eqref{eq:Hessian_accuracy_H_quadratic_SARC} \cite{cartis2011adaptiveI}.

\subsubsection{Merits of Condition \ref{condition:Hessian_approximation_H} \label{sec:merits}}
For our trust-region analysis, we require Condition \ref{condition:Hessian_approximation_H} with $ \epsilon \in \bigO(\max \left\{ \epsilon_{H}, \Delta_{t} \right\}) $; see \eqref{eq:STR_epsilon} in Theorem \ref{theorem:STR_main_det}. Hence, when $ \Delta_{t} $ is large, e.g., at the beginning of iterations,  all the conditions \eqref{eq:Hessian_accuracy_H}, \eqref{eq:Hessian_accuracy_H_quadratic_STR}, and \eqref{eq:Hessian_accuracy_H_quadratic_STR_new} are equivalent, up to some constants. However, the constants in \eqref{eq:Hessian_accuracy_H_quadratic_STR} and \eqref{eq:Hessian_accuracy_H_quadratic_STR_new} can be larger than what is implied by \eqref{eq:Hessian_accuracy_H}, amounting to cruder approximations in practice for when $ \Delta_{t} $ is large. As iterations progress, the trust-region radius will get smaller, and in fact it is expected that $ \Delta_{t} $ will eventually shrink to be $ \Delta_{t} \in \Theta \left(\min\{\epsilon_{g},\epsilon_{H}\}\right) $.  In prior works, e.g.,  \cite{blanchet2016convergence,gratton2017complexity}, the convergence analysis is derived using $ \epsilon_{H} = \epsilon_{g} $, whereas here we allow $ \epsilon_{H} = \sqrt{\epsilon_{g}}$. As a result, the requirements \eqref{eq:Hessian_accuracy_H_quadratic_STR} and \eqref{eq:Hessian_accuracy_H_quadratic_STR_new} can eventually amount to stricter conditions than  \eqref{eq:Hessian_accuracy_H}.
	
As for \eqref{eq:Hessian_accuracy_H_quadratic_SARC}, the main drawback lies in the difficulty of enforcing it. Despite the fact that for certain values of $ \|\s_{t}\| $ and $ \epsilon $, e.g., $ \epsilon \ll \|\s_{t}\| $, \eqref{eq:Hessian_accuracy_H_quadratic_SARC} can be less restrictive than \eqref{eq:Hessian_accuracy_H}, a priori enforcing \eqref{eq:Hessian_accuracy_H_quadratic_SARC} requires one to have already computed the search direction $\s_t$, which itself can be done only after $\H(\x_t)$ is constructed, hence creating a vicious circle. A posteriori guarantees can be given if one obtains a lower-bound estimate on the yet-to-be-computed step-size, i.e., to have  $ s_{0}>0 $  such that $ s_{0} \leq \|\s_{t}\| $. This allows one to consider a stronger condition as $ \norm{\left(\H(\x_{t}) - \nabla^2 F(\x_{t})\right)}\le C_{3} s_{0} $, which can be enforced using a variety of methods such as those described in Section \ref{sec:sub_sampling_and_finite_sum}. However, to obtain such a lower-bound estimate on the next step-size, one has to resort to a recursive procedure, which necessitates repeated constructions of the approximate Hessian and subsequent solutions of the corresponding subproblems. 
\red{Consequently, this procedure may result in a significant computational overhead and will lead to undesirable theoretical complexities.}
%Clearly, this procedure will result in a significant computational overhead and will lead to undesirable theoretical complexities.
	
In sharp contrast to \eqref{eq:Hessian_accuracy_H_quadratic_SARC}, the condition \eqref{eq:Hessian_accuracy_H} allows for theoretically principled use of many practical techniques to construct $ \H_{t} $. For example, under \eqref{eq:Hessian_accuracy_H}, the use of quasi-Newton methods to approximate the Hessian is theoretically justified. Further, by considering the stronger condition \eqref{eq:Hessian_accuracy_H2}, many \emph{randomized matrix approximation} techniques can be readily applied, e.g., \cite{woodruff2014sketching,mahoney2011randomized,tropp2015introduction,tropp_concentration}; see Section \ref{sec:sub_sampling_and_finite_sum}.
To the best of our knowledge, the only successful attempt at guaranteed a priori construction of $ \H_{t} $ using \eqref{eq:Hessian_accuracy_H_quadratic_SARC} is done in \cite{cartis2015global}. Specifically, by considering probabilistic models, which are ``sufficiently accurate'' in that they are partly based on \eqref{eq:Hessian_accuracy_H_quadratic_SARC},  \cite{cartis2015global} studies first-order complexity of a large class of methods, including ARC, and discusses ways to construct such probabilistic models as long as the gradient is large enough, i.e., before first-order approximate-optimality is achieved.  Here, by considering \eqref{eq:Hessian_accuracy_H}, we are able to provide an alternative analysis, which allows us to obtain second-order complexity results.

Requiring \eqref{eq:Hessian_accuracy_H2}, as a way of enforcing \eqref{eq:Hessian_accuracy_H}, offers a variety of other practical advantages, which are not readily available with other conditions. For example, consider distributed/parallel environments where the data is distributed across a network and the main bottleneck of computations is the communications across the nodes. In such settings, since \eqref{eq:Hessian_accuracy_H2} allows for the Hessian accuracy to be set a priori and to remain fixed across all iterations, the number of samples in each node can stay the same throughout iterations. This  prevents unnecessary communications to re-distribute the data at every iteration.

Furthermore, in case of failed iterations, i.e., when the computed steps are rejected, the previous $ \H_{t} $ may seamlessly be used in the next iteration, which avoids repeating many such, potentially expensive, computations throughout the iterations. For example, consider approximate solutions to the underlying sub-problems by means of dimensionality reduction, i.e., $  \H_{t}  $ is projected onto a lower dimensional sub-space as $  \U^{T}\H_{t} \U  $ for some $ \U \in \mathbb{R}^{d \times p} $ with $ p \ll d $, resulting in a smaller dimensional sub-problem. Now if the current iteration leads to a rejected step, the projection of the $  \H_{t}  $ from the previous iteration can be readily re-used in the next iteration. This naturally amounts to saving further Hessian computations.

\section{Algorithms and Convergence Analysis}
\label{sec:convergence_analysis}
%In this section, we consider the generic non-convex optimization \eqref{eq:obj} and study TR and CR methods with \emph{inexact} Hessian. 
We are now ready to present our main algorithms for solving the generic non-convex optimization \eqref{eq:obj} along with their corresponding iteration complexity results to obtain a $(\epsilon_g, \epsilon_H)$-optimal solution as in \eqref{eq:2ndopt}. More precisely, in Section \ref{sec:TR_analysis} and \ref{sec:ARC_analysis}, respectively, we present modifications of the TR and ARC methods which incorporate inexact Hessian information, according to Condition \ref{condition:Hessian_approximation_H}.

We remind that, though not specifically mentioned in the statement of the theorems or the algorithms, when the computed steps are rejected and an iteration needs to be repeated with different $ \Delta_{t} $ or $ \sigma_{t} $, the previous $ \H_{t} $ may seamlessly be used in the next iteration. This can be a desirable feature in many practical situations and is directly the result of enforcing \eqref{eq:Hessian_accuracy_H2}; see also the discussion in Section \ref{sec:merits}.

%As it can be seen, the bulk of the algorithms are almost identical to the case where the exact Hessian is used and the difference merely boils down to the use of $ \H(\x_{t}) $ as opposed to $ \nabla^{2} F(\x_{t}) $. As indicated in Section \ref{sec:notation}, for notational simplicity, in what follows, we use $ \H_{t} $ to denote $ \H(\x_{t}) $.

For our analysis throughout the paper, we make the following standard assumption regarding the regularity of the exact Hessian of the objective function $ F $.
\begin{assumption}[Hessian Regularity]
	\label{assumption:a1}
	$F(\x)$ is twice differentiable and has bounded and Lipschitz continuous Hessian on the piece-wise linear path generated by the iterates, i.e. for some $ 0 < K, L < \infty $ and all iterations 
	\begin{subequations}
		\label{eq:Hessian_regularity_F}
		\begin{align}
			&\norm{\nabla^2 F(\x) -\nabla^2 F(\x_{t})} \le L \norm{\x - \x_{t}}, \; \forall \x \in [\x_{t}, \x_{t} + \s_{t}], \label{eq:Hessian_Lipschitz_F}
			\\
			& \norm{\nabla^2 F(\x_{t})} \le K, 	\label{eq:Hessian_boundedness_F}
		\end{align}
	\end{subequations}
	where $ \x_{t} $ and $ \s_{t} $ are, respectively, the iterate and the update step at iteration $ t $.
	%	Note that we can relax the Lipschitz continuity assumption of the Hessian to hold only on the piece-wise linear path generated by the iterates, but for the sake of exposition, we stick with the global case.
\end{assumption}
Although, we do not know of a particular way to, a priori, verify \eqref{eq:Hessian_Lipschitz_F}, it is clear that  Assumption \eqref{eq:Hessian_Lipschitz_F} is weaker than Lipschitz continuity of the Hessian for all $ \x $, i.e.,
\begin{align}
	\norm{\nabla^2 F(\x) -\nabla^2 F(\y)} \le L \norm{\x - \y}, \; \forall \x, \y  \in \mathbb{R}^{d}. \label{eq:Hessian_Lipschitz_F_all}
\end{align}
Despite the fact that theoretically \eqref{eq:Hessian_Lipschitz_F} is weaker than \eqref{eq:Hessian_Lipschitz_F_all}, to the best of our knowledge as of yet, \eqref{eq:Hessian_Lipschitz_F_all} is the only practical sufficient condition for verifying \eqref{eq:Hessian_Lipschitz_F}. 

\subsection{Trust Region with Inexact Hessian}
\label{sec:TR_analysis}
Algorithm \ref{alg:STR_fg} depicts a trust-region algorithm where at each iteration $ t $, instead of the true Hessian $ \nabla^{2} F(\x_{t}) $, only an inexact approximation, $\H_{t} $, is used. 
For Algorithm \ref{alg:STR_fg}, the accuracy tolerance in \eqref{eq:Hessian_accuracy_H} is adaptively chosen as $  \epsilon_{t} \leq \max \left\{ \epsilon_{0}, \Delta_{t} \right\} $, where $ \Delta_{t} $ is the trust region in the t-th iteration and $ \epsilon_{0} \in \mathcal{O}(\epsilon_{H}) $ is some fixed threshold. This allows for a very crude approximation at the beginning of iterations, when $ \Delta_{t} $ is large. As iterations progress towards optimality and $ \Delta_{t} $ gets small, the threshold $ \epsilon_{0} $ can prevent $ \epsilon $ from getting unnecessarily too small.

\begin{algorithm}[htb]
	\caption{Trust Region with Inexact Hessian}
	\label{alg:STR_fg}
	\begin{algorithmic}[1]
		\STATE {\bf Input:} Starting point $\x_0$, initial radius $0 < \Delta_0 < \infty$, hyper-parameters $\epsilon_{0}, \epsilon_{g}, \epsilon_{H}, \eta\in(0,1), \gamma > 1$ 
		\FOR{$ t = 0,1,\ldots $}
		\STATE Set the approximate Hessian, $\H_t$, as in \eqref{eq:Hessian_approximation} with $  \epsilon_{t} \leq \max \left\{ \epsilon_{0}, \Delta_{t} \right\} $ \label{step:STR_step}
		\IF{$\norm{\nabla F(\x_{t})} \le \epsilon_g,  \lmin(\H_t) \ge -\epsilon_H\;$}  
		\STATE  Return $\x_t$.
		\ENDIF
		\STATE Solve the sub-problem approximately
		\begin{align}
		\label{eq:STR_subp}
		\s_t \approx \arg\min_{\|\s\|\le \Delta_t} m_t(s) \triangleq \lin{\nabla F(\x_{t}), \s} + \frac{1}{2}\lin{\s, \H_t \s} %\tag{$\ast$}
		\end{align}
		\STATE Set $\rho_t \triangleq \dfrac{F(\x_t) - F(\x_t + \s_t)}{-m_t(\s_t)}$
		\IF {$\rho_t \ge \eta$}
		\STATE $\x_{t+1} = \x_t + \s_t$
		\STATE $\Delta_{t+1} = \gamma \Delta_t$
		\ELSE
		\STATE $\x_{t+1} = \x_t$
		\STATE $\Delta_{t+1} = \Delta_t/\gamma$
		\ENDIF
		\ENDFOR
		\STATE {\bf Output:} $\x_{t}$
	\end{algorithmic}
\end{algorithm}

In Algorithm \ref{alg:STR_fg}, we require that the sub-problem \eqref{eq:STR_subp} is solved only approximately. Indeed, in large-scale problems, where the exact solution of the sub-problem is the main bottleneck of the computations, this is a very crucial relaxation. Such approximate solution of the sub-problem \eqref{eq:STR_subp} has been adopted in many previous work. Here, we follow the inexactness conditions discussed in \cite{conn2000trust}, which are widely known as Cauchy and Eigenpoint conditions. Recall that the Cauchy and Eigen directions correspond, respectively, to one dimensional minimization of the sub-problem \eqref{eq:STR_subp} along the directions given by the gradient and negative curvature. 
\begin{condition}[Sufficient Descent Cauchy and Eigen Directions \cite{conn2000trust}]
	\label{condition:STC_sufficient_descent} 
	Assume that we solve the sub-problem \eqref{eq:STR_subp} approximately to find $\s_t$ such that
	\begin{subequations}%{$ \textbf{S}_{_{\textbf{TR}}} $}
		\label{eq:STR_cond}
		\begin{align}
		&-m_t(\s_t) \ge - m_t(\s_t^C) \ge \frac{1}{2}\|\nabla F(\x_{t})\|\min\left\{\frac{\|\nabla F(\x_{t})\|}{1+\|\H_t\|}, \Delta_t\right\} \label{eq:STR_Cauchy},  \\
		&-m_t(\s_t) \ge - m_t(\s_t^E)\ge \frac{1}{2}\nu|\lambda_{\min}(\H_t)|\Delta_t^2, \quad \text{if}~ \lmin(\H_t) < 0. \label{eq:STR_Eigen} 
		\end{align}
	\end{subequations}
	Here, $m_t(\cdot)$ is defined in \eqref{eq:STR_subp}, $\s_t^C$ (Cauchy point) is along negative gradient direction and $\s_t^E$ is along approximate negative curvature direction such that $\lin{\s_t^E, \H_t \s_t^E} \le \nu\lmin(\H_t)\|\s_t^E\|^2 < 0$, for some $ \nu \in (0,1] $ (see Appendix B for a way to efficiently compute $ \s_{t}^{E} $).
\end{condition}
One way to ensure that an approximate solution to the sub-problem \eqref{eq:STR_subp} satisfies \eqref{eq:STR_cond}, is by replacing \eqref{eq:STR_subp} with the following reduced-dimension problem, in which the search space is a two-dimensional sub-space containing vectors $\s_t^C$, and $\s_t^E$, i.e.,
\begin{align*}
\s_t = \arg\min_{\stackrel{\|\s\|\le \Delta_t}{\s \in \text{Span}\{\s_t^C, \s_t^E\}}} \lin{\nabla F(\x_{t}), \s} + \frac{1}{2}\lin{\s, \H_t \s}.
\end{align*}
Of course, any larger dimensional sub-space $ \mathcal{P} $ for which we have $ \text{Span}\{\s_t^C, \s_t^E\} \subseteq \mathcal{P} $ would also guarantee \eqref{eq:STR_cond}. In fact, a larger dimensional sub-space implies a more accurate solution to our original sub-problem \eqref{eq:STR_subp}.

% Using the inexactness Condition \ref{condition:STC_sufficient_descent}, Theorem \ref{theorem:STR_main_det} gives the iteration complexity of Algorithm \ref{alg:STR_fg}, which as discussed in Section \ref{sec:TR}, has been shown to be optimal. The proof of Theorem \ref{theorem:STR_main_det} can be found in Appendix \ref{sec:TR_proofs}.

We now set out to provide iteration complexity for Algorithm \ref{alg:STR_fg}. Our analysis follows similar line of reasoning as that in \cite{cartis2011adaptiveI,cartis2011adaptiveII,cartis2012complexity}. First, we show the discrepancy between the quadratic model and objective function in Lemma \ref{lemma:fundecrease}.
% -------------------- % function approximation % -------------------- %
\begin{lemma}
\label{lemma:fundecrease}
Given Assumption \ref{assumption:a1} and Condition \eqref{eq:Hessian_accuracy_H} with any $ \epsilon_{t}  > 0 $, we have
\begin{align}\label{eq:STR_fundec}
|F(\x_t + \s_t) - F(\x_t) - m_t(\s_t) | \le \frac{L}{2}  \Delta_t^3 + \frac{ \epsilon_{t}}{2}\Delta_t^2.
\end{align}
\end{lemma}
\begin{proof}
Applying Mean Value Theorem on $F$ at $\x_t$ gives $F(\x_t + \s_t) = F(\x_t) + \nabla F(\x_t)^T \s_t + \frac{1}{2}\s_t^T \nabla^2F(\xi_t)\s_t$, for some $\xi_t$ in the segment of $[\x_t, \x_t + \s_t]$. We have 
\begin{align*}
&\Abs{F(\x_t + \s_t) - F(\x_t) - m_t(\s_t)} = \frac{1}{2}\Abs{\s_t^T(\nabla^2 F(\xi_t) - \H_t)\s_t}
\\ &= \frac{1}{2}\Abs{\s_t^T(\nabla^2 F(\xi_t) - \nabla^2 F(\x_t) + \nabla^2 F(\x_t) - \H_t)\s_t}
\\ &\le \frac{1}{2}\Abs{\s_t^T(\nabla^2 F(\xi_t) - \nabla^2 F(\x_t))\s_t} + \frac{1}{2}\Abs{\s_t^T(\nabla^2 F(\x_t) - \H_t)\s_t}
\\ &\le \frac{L}{2} \|\s_t\|^3 +  \frac{\epsilon_{t}}{2}\|\s_t\|^2 \le \frac{L}{2}\Delta_t^3 + \frac{\epsilon_{t}}{2}\Delta_t^2. \qed
\end{align*}
\end{proof}

Combining with Conditions \ref{condition:Hessian_approximation_H} and \ref{condition:STC_sufficient_descent},  we get the following two lemmas that characterize sufficient conditions for successful iterations. 
% -------------------- % eigen direction % -------------------- %
\begin{lemma}
\label{lemma:model_eig}
Consider any $ \epsilon_{H} > 0 $, let $ \epsilon_{0} \triangleq \alpha (1-\eta)\nu \epsilon_{H} $ for some $ \alpha \in (0,1)$, and suppose Condition \ref{condition:Hessian_approximation_H} is satisfied with $\epsilon_{t} \leq \max\{\epsilon_{0}, \Delta_{t}\} $, where $ \Delta_{t} $ is the trust region at the t-th iteration. Given Assumption \ref{assumption:a1} and Condition \ref{condition:STC_sufficient_descent}, if $\lmin(\H_t) < -\epsilon_{H}$ and $\Delta_t \le (1-\alpha)(1-\eta)\nu \Abs{\lmin(\H_t)} / (L+1)$, then the t-th iteration is {\it successful}, i.e. $\Delta_{t+1} = \gamma \Delta_t$.
\end{lemma}
\begin{proof}
Suppose $ \Delta_{t} \leq \epsilon_{0} $. From \eqref{eq:STR_Eigen} and \eqref{eq:STR_fundec}, we have
\begin{align*}
1 - \rho_t &= \frac{F(\x_t + \s_t) - F(\x_t)-m_t(\s_t)}{-m_t(\s_t)}  \le \frac{L \Delta_t^3 +  \epsilon_{t}\Delta_t^2}{\nu\Abs{\lmin(\H_t)}\Delta_t^2} \leq \frac{L \Delta_t^3 +  \alpha (1-\eta)\nu \epsilon_{H}\Delta_t^2}{\nu\Abs{\lmin(\H_t)}\Delta_t^2} \\
& \le \frac{L \Delta_t^3 + \alpha (1-\eta)\nu  \Abs{\lmin(\H_t)}\Delta_t^2}{\nu \Abs{\lmin(\H_t)} \Delta_t^2} \le \frac{L \Delta_t + \alpha (1-\eta)\nu \Abs{\lmin(\H_t)}}{\nu \Abs{\lmin(\H_t)}}.
\end{align*}
By the assumption on $\Delta_t$, we get $\rho_t \ge \eta$ and the iteration is successful. 
Now consider $ \Delta_{t} \geq \epsilon_{0} $. Similar to the above, we have 
\begin{align*}
	1 - \rho_t &= \frac{F(\x_t + \s_t) - F(\x_t)-m_t(\s_t)}{-m_t(\s_t)}  \le \frac{L \Delta_t^3 +  \epsilon_{t}\Delta_t^2}{\nu\Abs{\lmin(\H_t)}\Delta_t^2} \\
	& \le \frac{(L + 1) \Delta_t^3}{\nu \Abs{\lmin(\H_t)} \Delta_t^2} \le \frac{(L + 1) \Delta_t}{\nu \Abs{\lmin(\H_t)}},
\end{align*}
which again by assumption on $ \Delta_{t} $ and noting $ \alpha < 1 $, we get $\rho_t \ge \eta$ and the iteration is successful. 
\qed
\end{proof}

% -------------------- % gradient direction % -------------------- %
\begin{lemma}%{lemma}{lemmodel_g2}
\label{lemma:model_g2}
Suppose Condition \ref{condition:Hessian_approximation_H} is satisfied with any $\epsilon_{t} > 0$.
Given Assumption \ref{assumption:a1} and Condition \ref{condition:STC_sufficient_descent}, if  $\|\nabla F(\x_{t})\| > \epsilon_{g}$ and
\begin{align*}
\Delta_t \le \min\left\{\frac{\|\nabla F(\x_{t})\|}{(1+K_{H})}, \frac{\sqrt{ \epsilon_{t}^{2} + 4 L (1-\eta) \|\nabla F(\x_{t})\|}- \epsilon_{t}}{2 L}\right\}, 
%\Delta_t \le \min\left\{\frac{\epsilon_{g}}{(1+K_{H})}, \refone{\frac{2(1-\eta)\epsilon_{g}}{3}, \frac{1}{2 L}, \frac{(1-\eta) \epsilon_{g}}{\sqrt{\epsilon + 4L(1-\eta) \epsilon_g}}}\right\}, 
\end{align*}
then, the t-th iteration is successful, i.e. $\Delta_{t+1} = \gamma \Delta_t$.
\end{lemma}
\begin{proof}
By assumption on $\Delta_{t}$, \eqref{eq:STR_Cauchy}, and since $\|\nabla F(\x_{t})\| > \epsilon_{g}$, we have
\begin{align*}
-m_t(\s_t) &\ge \frac{1}{2}\|\nabla F(\x_{t})\|\min\left\{\frac{\|\nabla F(\x_{t})\|}{1+\|\H_t\|}, \Delta_t\right\}
%\ge \frac{1}{2}\epsilon_{g}\min\left\{\frac{\epsilon_{g}}{1+K_{H}}, \Delta_t\right\} 
\ge \frac{1}{2}\|\nabla F(\x_{t})\|\Delta_t.
\end{align*}
Therefore,
\begin{align*}
1 - \rho_t  & = \frac{F(\x_t + \s_t) - F(\x_t) - m_t(\s_t)}{- m_t(\s_t)} \leq \frac{L\Delta_t^3 +  \epsilon_{t}\Delta_t^2}{\|\nabla F(\x_{t})\|\Delta_t} \leq \frac{L\Delta_t^2 +  \epsilon_{t}\Delta_t}{\|\nabla F(\x_{t})\|} \leq 1-\eta,
\end{align*}
where the last inequality follows by assumption on $ \Delta_{t} $.
%\refone{If $4L(1 - \eta)\epsilon_g \ge 3$, then $\Delta_t \le {1}/{(2L)} \le {(-1 + \sqrt{1 + 4L(1-\eta)\epsilon_g)}}/{(2L)}$. Otherwise, $\Delta_t \le {2(1 - \eta)\epsilon_g}/{3} \le {(-1 + \sqrt{1 + 4L(1-\eta)\epsilon_g})}/{(2L)}$. In either case, since the function the function $ h(t) \triangleq {-t + \sqrt{t^2 + a}} $, for any $ a > 0 $, is decreasing in $ t $, we have
%\begin{align*}
%	\Delta_t \leq \frac{-1 + \sqrt{1 + 4L(1-\eta) \epsilon_g}}{2L} \le \frac{-\epsilon + \sqrt{\epsilon + 4L(1-\eta) \epsilon_g}}{2L}.
%\end{align*}
%Now by assumption on $ \Delta_{t} $, we have $ \Delta_{t} \sqrt{\epsilon + 4L(1-\eta) \epsilon_g} \leq (1-\eta) \epsilon_{g} $, which implies $L\Delta_t^2 + \epsilon \Delta_t \le (1 - \eta) \epsilon_{g}$.} 
So $\rho_t \ge \eta$, which means the iteration is successful. \qed
\end{proof}

Lemma \ref{lemma:lowerbound_delta} gives a lower bound for the trust region radius before the algorithm terminates, i.e., this ensures that the trust region never shrinks to become too small.
% -------------------- % lower bound delta % -------------------- %
\begin{lemma}
		\label{lemma:lowerbound_delta}
		Consider any $ \epsilon_{g}, \epsilon_{H} > 0 $ such that $ \epsilon_{H} \leq \sqrt{\epsilon_{g}} $ and let $ \epsilon_{0} \triangleq \alpha (1-\eta)\nu \epsilon_{H} $ for some $ \alpha \in (0,1)$. Further, suppose Condition \ref{condition:Hessian_approximation_H} is satisfied with $\epsilon_{t} \leq \max\{\epsilon_{0}, \Delta_{t}\} $, where $ \Delta_{t} $ is the trust region at the t-th iteration. For Algorithm \ref{alg:STR_fg}, under Assumption \ref{assumption:a1} and Condition \ref{condition:STC_sufficient_descent}, we have $\Delta_t \ge \kappa_{\Delta} \min\{\epsilon_{g},\epsilon_{H}\}, \forall t \geq 0$, where 
		\begin{align*}
			\kappa_{\Delta} &\triangleq \min\left\{\kappa_{1}, \kappa_{2}, \kappa_{3}, \kappa_{4}\right\} / \gamma, \quad \kappa_{1} \triangleq {(1-\alpha)(1-\eta)\nu}/{(L+1)}, \quad \kappa_{2} \triangleq \alpha (1-\eta)\nu, \\
			\kappa_{3} &\triangleq {1}/{(1+K_{H})}, \quad \kappa_{4} \triangleq {\sqrt{(\alpha (1-\eta)\nu)^{2} + 4 L (1-\eta) }-\alpha (1-\eta)\nu}/{(2 L)}.
		\end{align*}
	\end{lemma}	
	\begin{proof}
		We prove by contradiction. Assume that the t-th iteration is the first unsuccessful iteration such that
		$\Delta_{t+1} = \Delta_{t}/\gamma  \le \kappa_{\Delta} \min\{\epsilon_{g},\epsilon_{H}\}$, i.e., we have 
		\begin{align*}
			\Delta_{t} \leq \min\left\{\kappa_{1}, \kappa_{2}, \kappa_{3}, \kappa_{4}\right\} \cdot \min\{\epsilon_{g},\epsilon_{H}\}.
		\end{align*}
		Suppose $\lambda_{\min}(\H_{t}) < -\epsilon_{H}$. By Lemma \ref{lemma:model_eig}, since $ \Delta_t \le (1-\alpha)(1-\eta)\nu \Abs{\lambda_{\min}(\H_{t})}/{(L+1)} $, iteration $ t $ must have been accepted and we must have $ \Delta_{t+1} = \gamma \Delta_{t} > \Delta_{t}$, which is a contradiction. 
		Now suppose $ \|\nabla F(\x_{t})\| \geq \epsilon_{g}$. By assumption on $ \Delta_{t} $, we have that $ \Delta_{t} \leq \kappa_{2} \epsilon_H = \epsilon_{0} $, which implies that $ \epsilon_{t} \leq \epsilon_{0} $. Since the function $ h(a,b) \triangleq {-a + \sqrt{a^2 + b}} $, for any fixed $ b > 0 $, is decreasing in $ a $,  and for any fixed $ a $, is increasing in $ b \geq 0$, we have
		\begin{align*}
			h( \epsilon_{t},4 L (1-\eta) \|\nabla F(\x_{t})\|) &\geq h(\epsilon_{0},4 L (1-\eta) \|\nabla F(\x_{t})\|) \\
			&\geq h(\epsilon_{0},4 L (1-\eta) \epsilon_{g}) \geq h(\epsilon_{0},4 L (1-\eta) \epsilon_{H}^{2}),
		\end{align*}
		which implies
		\begin{align*}
			\frac{\sqrt{ \epsilon_{t}^{2} + 4 L (1-\eta) \|\nabla F(\x_{t})\|}- \epsilon_{t}}{2L} \geq \kappa_{4} \epsilon_H.
		\end{align*}
		As a result, since $ \Delta_{t} \leq \min\{\kappa_{3} \epsilon_g, \kappa_{4} \epsilon_H\} $, it must satisfy the condition of Lemma \ref{lemma:model_g2}. This implies that iteration $ t $ must have been accepted, which is a contradiction. 		\qed
\end{proof}

The following lemma follows closely the line of reasoning in \cite[Lemma 4.5]{cartis2012complexity}.
% -------------------- % succese iterations % -------------------- %
\begin{lemma}[Successful Iterations]
\label{lemma:STR_succ}
Consider any $ \epsilon_{g}, \epsilon_{H} > 0 $ such that $ \epsilon_{H} \leq \sqrt{\epsilon_{g}} $ and let $ \epsilon_{0} \triangleq \alpha (1-\eta)\nu \epsilon_{H} $ for some $ \alpha \in (0,1)$. Further, suppose Condition \ref{condition:Hessian_approximation_H} is satisfied with $\epsilon_{t} \leq \max\{\epsilon_{0}, \Delta_{t}\} $, where $ \Delta_{t} $ is the trust region at the t-th iteration. Let $\mathcal{T}_{succ}$ denote the set of all the successful iterations before Algorithm \ref{alg:STR_fg} stops. Then, under Assumption \ref{assumption:a1}, Condition \ref{condition:STC_sufficient_descent}, the number of successful iterations is upper bounded by,
$$\Abs{\mathcal{T}_{succ}} \le \frac{(F(\x_0) - F_{\min})}{\eta \min\left\{\widehat{\kappa}_\Delta, \widetilde{\kappa}_{\Delta}\right\} } \cdot \max\{\epsilon_g^{-2}\epsilon_H^{-1}, \epsilon_H^{-3}\} $$ 
where $\widehat{\kappa}_\Delta\triangleq  \kappa_\Delta/2, \widetilde{\kappa}_{\Delta}  \triangleq \nu\kappa_\Delta^2/2$, and $ \kappa_{\Delta} $ is as defined in Lemma \ref{lemma:lowerbound_delta}.
\end{lemma}
\begin{proof}
Suppose Algorithm \ref{alg:STR_fg} doesn't terminate at the t-th iteration. Then we have either $\norm{\nabla F(\x_{t})} \ge \epsilon_g$ or $\lmin(\Delta^2 F(\x_t)) \le -\epsilon_{H}$. In the first case, %where $\norm{\nabla F(\x_{t})} \ge \epsilon_g$, 
from \eqref{eq:STR_Cauchy}, we have
\begin{align*}
-m_t(\s_t) & \ge \frac{\epsilon_g}{2} \min\left\{\frac{\epsilon_g}{1+K_{H}}, \Delta_t\right\}\ge \frac{\epsilon_g}{2} \min\left\{\frac{\epsilon_g}{1+K_{H}}, \kappa_\Delta \epsilon_g, \kappa_\Delta \epsilon_H\right\}
\geq \widehat{\kappa}_\Delta \epsilon_g \min\{\epsilon_g, \epsilon_H\},
\end{align*}
where $ \kappa_{\Delta} $ is as defined in Lemma \ref{lemma:lowerbound_delta}.
%where
%\begin{align*}
%$\widehat{\kappa}_\Delta \triangleq \frac{1}{2} \min\left\{\frac{1}{1+K_{H}}, \kappa_\Delta\right\}$. 
%\end{align*}
Similarly, in the second case, from \eqref{eq:STR_Eigen}, we obtain
\begin{align*}
-m_t(\s_t) \ge \frac{1}{2}\nu\Abs{\lmin(\H_t)}\Delta_t^2 \ge \frac{1}{2}\nu\kappa_\Delta^2 \epsilon_H\min\{\epsilon_g^2, \epsilon_H^2\} = \widetilde{\kappa}_{\Delta} \epsilon_H\min\{\epsilon_g^2, \epsilon_H^2\}.
\end{align*}
%where
%\begin{align*}
%$\widetilde{\kappa}_{\Delta}  \triangleq \frac{1}{2}\nu\kappa_\Delta^2$.
%\end{align*}
Since $F(\x_t)$ is monotonically decreasing, we have
\begin{align*}
F(\x_0) - F_{\min} &\ge \sum_{t=0}^\infty F(\x_t) - F(\x_{t+1}) \ge \sum_{t\in \mathcal{T}_{succ}} F(\x_t) - F(\x_{t+1})
\\ & \ge \eta \sum_{t\in \mathcal{T}_{succ}} \min\left\{\widehat{\kappa}_\Delta \epsilon_g \min\{\epsilon_g, \epsilon_H\}, \widetilde{\kappa}_{\Delta} \epsilon_H\min\{\epsilon_g^2, \epsilon_H^2\}\right\}
\\ & \geq \Abs{\mathcal{T}_{succ}} \eta \min\left\{\widehat{\kappa}_\Delta, \widetilde{\kappa}_{\Delta}\right\}  \min\{\epsilon_g^2 \epsilon_H, \epsilon_H^3\}.
\end{align*}
Hence, we have $\Abs{\mathcal{T}_{succ}} \le (F(\x_0) - F_{\min}) \max\{\epsilon_g^{-2}\epsilon_H^{-1}, \epsilon_H^{-3}\} / (\eta \min\left\{\widehat{\kappa}_\Delta, \widetilde{\kappa}_{\Delta}\right\} )$. \qed
\end{proof}

Now we are ready to present the final complexity in Theorem \ref{theorem:STR_main_det}.
\begin{theorem}[Optimal Complexity of Algorithm \ref{alg:STR_fg}]%{theorem}{theoremSTC}
	\label{theorem:STR_main_det}
	Consider any $ \epsilon_{g},\epsilon_{H} > 0$ such that $ \epsilon_{H} \leq \sqrt{\epsilon_{g}} $ and let $ \epsilon_{0} \triangleq \alpha (1-\eta)\nu \epsilon_{H} $ for some $ \alpha \in (0,1)$ where $ \eta $ is a hyper-parameter in Algorithm \ref{alg:STR_fg}, and $ \nu$ is as in \eqref{eq:STR_Eigen}. Suppose the inexact Hessian, $ \H(\x)$, satisfies Condition \ref{condition:Hessian_approximation_H} with the approximation tolerance, $  \epsilon_{t} $, in \eqref{eq:Hessian_accuracy_H} as 
	\begin{align}
	\label{eq:STR_epsilon}
	%\epsilon \in (0, \alpha (1-\eta)\nu\cdot \epsilon_H),
%	\epsilon < (1-\eta)\nu \epsilon_H,
	\epsilon_{t} \leq \max \left\{ \epsilon_{0}, \Delta_{t} \right\},
	\end{align}
	where $ \Delta_{t} $ is the trust region at the t-th iteration. For Problem \eqref{eq:obj}, under Assumption \ref{assumption:a1} and Condition \ref{condition:STC_sufficient_descent}, Algorithm \ref{alg:STR_fg} terminates after at most 
%	\begin{align*}
	$T \in \bigO\left(\max\{\epsilon_g^{-2}\epsilon_H^{-1}, \epsilon_H^{-3}\}\right)$
%	\end{align*}
	iterations.
\end{theorem}
\begin{proof}
Suppose Algorithm \ref{alg:STR_fg} terminates at the t-th iteration.
Let $\mathcal{T}_{succ}$ and $\mathcal{T}_{fail}$ denote the sets of all the successful and unsuccessful iterations, respectively. Then $T = \Abs{\mathcal{T}_{succ}} + \Abs{\mathcal{T}_{fail}}$ and $\Delta_T = \Delta_0 \gamma^{\Abs{\mathcal{T}_{succ}} - \Abs{\mathcal{T}_{fail}}}$, where $ \gamma $ is a hyper-parameter of Algorithm \ref{alg:STR_fg}. From Lemma \ref{lemma:lowerbound_delta}, we have $\Delta_T \ge \kappa_\Delta \min\{\epsilon_g, \epsilon_H\}$. Hence, $\left(\Abs{\mathcal{T}_{succ}} - \Abs{\mathcal{T}_{fail}}\right) \log \gamma \ge \log\left( \kappa_\Delta \cdot \min\{\epsilon_g, \epsilon_H\}/\Delta_0\right)$,
which implies $\Abs{\mathcal{T}_{fail}} \le \log\left( \Delta_0/\left(\kappa_\Delta\cdot \min\{\epsilon_g, \epsilon_H\}\right)\right)/\log \gamma + \Abs{\mathcal{T}_{succ}}$.
Combine the result from Lemma \ref{lemma:STR_succ}, we have the total iteration complexity as
\begin{align*}
T &\le \frac{1}{\log \gamma}\log\left( \frac{\Delta_0} {\kappa_\Delta\cdot \min\{\epsilon_g, \epsilon_H\}}\right) + \frac{2(F(\x_0) - F_{\min})}{\eta \min\left\{\widehat{\kappa}_\Delta, \widetilde{\kappa}_{\Delta}\right\} } \cdot \max\{\epsilon_g^{-2}\epsilon_H^{-1}, \epsilon_H^{-3}\} 
\\&\in \bigO\left(\max\{\epsilon_g^{-2}\epsilon_H^{-1}, \epsilon_H^{-3}\}\right),
\end{align*}
where $ \kappa_\Delta, \widehat{\kappa}_\Delta, \widetilde{\kappa}_{\Delta}$ are defined in the proofs of Lemmas \ref{lemma:lowerbound_delta} and  \ref{lemma:STR_succ}, respectively.
\qed
\end{proof}

As it can be seen, the worst-case total number of iterations required by Algorithm \ref{alg:STR_fg} before termination, matches the optimal iteration complexity obtained in \cite{cartis2012complexity}. Furthermore, from \eqref{eq:Hessian_accuracy_H}, it follows that upon termination of Algorithm \ref{alg:STR_fg} after $ T $ iterations, in addition to $ \norm{\nabla F(\x_{T})} \leq \epsilon_{g} $, we have $ \lambda_{\min}\left(\nabla^{2} F(\x_{T})\right) \geq -(\epsilon_{H} + \epsilon_{T}) $, i.e., the obtained solution satisfies $(\epsilon_g, \epsilon_{T} + \epsilon_H)$-Optimality as in \eqref{eq:2ndopt}.
	
For Algorithm \ref{alg:STR_fg}, the Hessian approximation tolerance $ \epsilon_{t} $ is allowed to be chosen per-iteration as $ \epsilon_{t} \leq \mathcal{O}\left(\max\{\epsilon_{H},\Delta_{t}\}\right) $. This way, when $ \Delta_{t} $ is large (e.g., at the beginning of iterations), one can employ crude Hessian approximations. As iterations progress towards optimality, $ \Delta_{t} $ can get very small, in which case Hessian accuracy is set in the order of $ \epsilon_{H} $. Note that by Lemma \ref{lemma:lowerbound_delta}, we are always guaranteed to have $ \Delta_{t} \in \Omega\left(\min \left\{\epsilon_{g}, \epsilon_{H} \right\}\right) $. As a result, when $ \epsilon_{g} \ll \epsilon_{H} $, e.g.,  $ \epsilon_{H}^{2} = \epsilon_{g} = \epsilon$, we can have that $ \Delta_{t} \ll \epsilon_{H}$. In such cases, the choice $ \epsilon_{t} \leq \mathcal{O}\left(\max\{\epsilon_{H},\Delta_{t}\}\right) $ ensures that the Hessian approximation tolerance never gets unnecessarily too small.
%As a result, our new modified analysis can offer the best of both worlds in cases where $ \epsilon_{g} \ll \epsilon_{H} $, and in other cases, at least as good as similar prior work.
%For example, a computationally efficient approach for tuning $ \epsilon_{t} $ such that at termination we have $ \lambda_{\min}\left(\nabla^{2} F(\x_{T})\right) \in \Omega (-\epsilon_{H}) $ is to start with large values, e.g., at the beginning of iterations when $ \Delta_{t} $ is large, we can set $ \epsilon_{t} \in \mathcal{O}(\Delta_{t}) $, and then, as iterations progress, we can gradually increase the Hessian approximation accuracy such that $ \epsilon_{t} \in \mathcal{O}(\epsilon_{H}) $ whenever $ \Delta_{t} \in \Omega(\epsilon_{H}) $.

% -----------------------------------------------------------------------
%
% -----------------------------------------------------------------------

\subsection{Adaptive Cubic Regularization with Inexact Hessian}
\label{sec:ARC_analysis}
Similar to Section \ref{sec:TR_analysis}, in this section, we present the algorithm and its corresponding convergence results for the case of adaptive cubic regularization with inexact Hessian. In particular, Algorithm \ref{alg:SARC_fg} depicts a variant of ARC algorithm where at each iteration $ t $, the inexact approximation, $\H_{t} $, is constructed according to Condition \ref{condition:Hessian_approximation_H}.
Here, unlike Section \ref{sec:TR_analysis}, we were unable to provide convergence guarantees with adaptive tolerance in \eqref{eq:Hessian_accuracy_H} and as result, $ \epsilon $ is set fixed a priori to a sufficiently small value, i.e., $ \epsilon \in \mathcal{O}(\sqrt{\epsilon_{g}}, \epsilon_{H})$ to guarantee $(\epsilon_{g}, \epsilon_{H})$-optimality.

\begin{algorithm}[!htbp]
	\caption{Adaptive Cubic Regularization with Inexact Hessian}
	\label{alg:SARC_fg}
	\begin{algorithmic}[1]
		\STATE {\bf Input:} Starting point $\x_0$, initial regularization $0 < \sigma_0 < \infty$, hyper-parameters $\epsilon_{g}, \epsilon_{H}, \eta\in(0,1), \gamma > 1$ 
		\FOR{$ t = 0,1,\ldots $}
		\STATE Set the approximating Hessian, $ \H_t$, as in \eqref{eq:Hessian_approximation} \label{step:SARC_step}  
		\IF{$\|\nabla F(\x_{t})\| \le \epsilon_g,  \lmin(\H_t) \ge -\epsilon_H\;$}
		\STATE  Return $\x_t$.
		\ENDIF
		\STATE Solve the sub-problem approximately
		\begin{align}
		\label{eq:SARC_subp}
		\s_t \approx \arg\min_{\s \in \mathbb{R}^{d}}~~ m_t(s) \triangleq \lin{\nabla F(\x_{t}), \s} + \frac{1}{2} \lin{\s, \H_t \s} + \frac{\sigma_{t}}{3} \|\s\|^{3}  %\tag{$\ast\ast$}
		\end{align}
		\STATE Set $\rho_t \triangleq \dfrac{F(\x_t) - F(\x_t + \s_t)}{-m_t(\s_t)}$
		\IF {$\rho_t \ge \eta$}
		\STATE $\x_{t+1} = \x_t + \s_t$
		\STATE $\sigma_{t+1} = \sigma_t/\gamma$
		\ELSE
		\STATE $\x_{t+1} = \x_t$
		\STATE $\sigma_{t+1} = \gamma \sigma_t$
		\ENDIF
		\ENDFOR
		\STATE {\bf Output:} $\x_{t}$
	\end{algorithmic}
\end{algorithm}

Similar to Algorithm \ref{alg:STR_fg}, here we also require that the sub-problem \eqref{eq:SARC_subp} in Algorithm \ref{alg:SARC_fg} is solved only approximately. Although similar inexact solutions to the sub-problem \eqref{eq:SARC_subp} by using Cauchy and Eigenpoint has been considered in several previous work, e.g., \cite{cartis2012complexity}, here we provide refined conditions which prove to be instrumental in obtaining iteration complexities with the relaxed Hessian approximation \eqref{eq:Hessian_accuracy_H}, as opposed to the stronger Condition \eqref{eq:Hessian_accuracy_H_quadratic_SARC}. 
%For the proof that the Cauchy and Eigenpoint, indeed satisfy the inequalities \eqref{eq:SARC_cond}, see Appendix \ref{sec:SARC_proofs}. 

\begin{condition}[Sufficient Descent Cauchy \& Eigen Directions]
	\label{condition:SARC_sufficient_descent} 
	Assume that we solve the sub-problem \eqref{eq:SARC_subp} approximately to find $\s_{t}$ such that
	\begin{subequations}%{$ \textbf{S}_{_{\textbf{CR}}} $}
		\label{eq:SARC_cond}
		\begin{align}
		&-m_t(\s_{t}) \ge - m_t(\s_t^C) \ge \max \Bigg\{ \frac{1}{12}  \|\s_t^{C}\|^{2} \left( \sqrt{K_{H}^{2} + 4 \sigma_{t} \|\nabla F(\x_{t})\|} - K_{H}\right), \Bigg. \nonumber\\
		& \quad \quad \quad \quad \quad \quad \quad \quad \quad \quad  \quad \quad \quad \Bigg.\frac{\|\nabla F(\x_{t})\|}{2 \sqrt{3} } \min \Bigg\{\frac{\|\nabla F(\x_{t}) \|}{K_{H}} , \sqrt{\frac{\|\nabla F(\x_{t}) \|}{\sigma_{t}}} \Bigg\} \Bigg\}, \label{eq:SARC_Cauchy} \\
		&-m_t(\s_{t}) \ge - m_t(\s_t^E) \geq \frac{\nu |\lambda_{\min}(\H_{t})|}{6}  \max \left\{ \|\s_{t}^{E}\|^{2} , \frac{\nu^{2} |\lambda_{\min}(\H_{t})|^{2}}{\sigma_{t}^{2}} \right\}, \; \text{if} \lmin(\H_t) < 0. \label{eq:SARC_Eigen}
		\end{align}
	\end{subequations}
	Here $m_t(\cdot)$ is defined in \eqref{eq:SARC_subp}, $\s_t^C$ (Cauchy point) is along negative gradient direction and $\s_t^E$ is along approximate negative curvature direction such that $\lin{\s_t^E, \H_t \s_t^E} \le \nu\lmin(\H_t)\|\s_t^E\|^2 < 0$ for some $ \nu \in (0,1] $ (see Appendix B for a way to efficiently compute $ \s_{t}^{E} $).
\end{condition}
Note that Condition \eqref{eq:SARC_cond} describes the quality of the descent obtained by Cauchy and Eigen directions more accurately than is usually found in similar literature. A natural way to ensure that the approximate solution to the sub-problem \eqref{eq:SARC_subp} satisfies \eqref{eq:SARC_cond}, is by replacing the unconstrained high-dimensional sub-problem \eqref{eq:SARC_subp} with the following constrained but lower-dimensional problem, in which the search space is reduced to a two-dimensional sub-space containing vectors $\s_t^C$, and $\s_t^E$, i.e.,
\begin{align*}
\s_t = \arg\min_{\s \in \text{Span}\{\s_t^C, \s_t^E\}} \lin{\nabla F(\x_{t}), \s} + \frac{1}{2}\lin{\s, \H_t \s} + \frac{\sigma_{t}}{3} \|\s\|^{3}.
\end{align*}
Note that, if $ \U \in \mathbb{R}^{d \times p} $ is an orthogonal basis for the sub-space ``$ \text{Span}\{\s_t^C, \s_t^E\} $'', by a linear transformation, we can turn the above sub-problem into an unconstrained problem as
\begin{align*}
\vv{v}_t = \arg\min_{\vv{v} \in \mathbb{R}^{p}} \lin{U^{T} \nabla F(\x_{t}), \vv{v}} + \frac{1}{2}\lin{\vv{v}, \U^{T} \H_t \U \vv{v}} + \frac{\sigma_{t}}{3} \|\vv{v}\|^{3},
\end{align*}
and set $ \s_{t} = \U \vv{v}_{t} $. As before, any larger dimensional sub-space $ \mathcal{P} $ for which we have $ \text{Span}\{\s_t^C, \s_t^E\} \subseteq \mathcal{P} $ would also ensure \eqref{eq:SARC_cond}, and, indeed, implies a more accurate solution to our original sub-problem \eqref{eq:SARC_subp}. 

% Theorem \ref{theorem:SARC_main_det} gives the iteration complexity of Algorithm \ref{alg:SARC_fg} for the case where the approximate solution to the sub-problem \eqref{eq:SARC_subp} is only required to satisfy the inexactness Condition \ref{condition:SARC_sufficient_descent}. The proof of Theorem \ref{theorem:SARC_main_det} can be found in Appendix \ref{sec:SARC_proofs}.

Lemmas \ref{lemma:SARC_Cauchy} and \ref{lemma:SARC_Eigen} describe the model reduction obtained by Cauchy and eigen points as required by Condition \eqref{condition:SARC_sufficient_descent}.
%We first show the quality of Cauchy direction and negative curvature direction for solving the subproblem approximately, as in Lemma \ref{lemma:SARC_Cauchy} \& \ref{lemma:SARC_Eigen}.
% -------------------- % Descent with Cauchy Direction % -------------------- %
\begin{lemma}[Descent with Cauchy Direction]
	\label{lemma:SARC_Cauchy}
	Consider the Cauchy direction as 
	$\s_{t}^{C}  = - \alpha \nabla F(\x_{t})$ where $\alpha = \arg\min_{\widehat{\alpha} \geq 0} m_{t}(- \widehat{\alpha} \nabla F(\x_{t}))$.
	 We have
	\begin{align*}
	& - m_t(\s_t^C) \ge \max \Bigg\{ \frac{1}{12}  \|\s_t^{C}\|^{2} \left( \sqrt{K_{H}^{2} + 4 \sigma_{t} \|\nabla F(\x_{t})\|} - K_{H}\right), \Bigg. \nonumber\\
	& \quad \quad \quad \quad \quad  \quad \quad \quad \Bigg.\frac{\|\nabla F(\x_{t})\|}{2 \sqrt{3} } \min \Bigg\{\frac{\|\nabla F(\x_{t}) \|}{K_{H}} , \sqrt{\frac{\|\nabla F(\x_{t}) \|}{\sigma_{t}}} \Bigg\} \Bigg\}.
%	- m_t(\s_t^C) &\ge \max\left\{ \frac{1}{6} \sigma_{t} \|\s_t^{C}\|^{2} \left( \sqrt{K_{H}^{2} + 4 \sigma_{t} \|\nabla F(\x_{t})\|} - K_{H}\right), \frac{4 \|\nabla F(\x_{t})\|}{6 \sqrt{3} } \min \left\{\frac{\|\nabla F(\x_{t}) \|}{K_{H}} , \sqrt{\frac{\|\nabla F(\x_{t}) \|}{\sigma_{t}}} \right\} \right\}.
	\end{align*}
\end{lemma}
\begin{proof}
For any $ \widehat{\alpha} \geq0 $, we have $ m_{t}(- \widehat{\alpha} \nabla F(\x_{t})) \leq m_{t}(\widehat{\alpha} \nabla F(\x_{t})) $, which implies $\alpha = \arg\min_{\widehat{\alpha} \in \mathbb{R}} m_{t}(- \widehat{\alpha} \nabla F(\x_{t}))$. Hence, we have $- \|\nabla F(\x_{t})\|^{2} + \alpha \lin{\nabla F(\x_{t}), \H_{t} \nabla F(\x_{t})} + \sigma_{t} \alpha^{2} \|\nabla F(\x_{t})\|^{3} = 0$.
We can find explicit formula for such $ \alpha $ by finding the roots of the quadratic function $r(\alpha) = \sigma_{t} \|\nabla F(\x_{t})\|^{3} \alpha^{2} + \lin{\nabla F(\x_{t}), \H_{t} \nabla F(\x_{t})} \alpha - \|\nabla F(\x_{t})\|^{2}$.
Hence, we must have
\begin{align*}
\alpha &= \frac{- \lin{\nabla F(\x_{t}), \H_{t} \nabla F(\x_{t})} + \sqrt{\big(\lin{\nabla F(\x_{t}), \H_{t} \nabla F(\x_{t})}\big)^{2} + 4 \sigma_{t} \|\nabla F(\x_{t})\|^{5}}}{2 \sigma_{t} \|\nabla F(\x_{t})\|^{3}} \geq 0.
\end{align*}
It follows that
\begin{align*}
2\alpha \sigma_t\|\nabla F(\x_{t})\| &= \sqrt{\left(\frac{\lin{\nabla F(\x_{t}), \H_{t} \nabla F(\x_{t})}}{\|\nabla F(\x_{t})\|^{2}}\right)^{2} + 4 \sigma_{t} \|\nabla F(\x_{t})\|} - \frac{\lin{\nabla F(\x_{t}), \H_{t} \nabla F(\x_{t})}}{\|\nabla F(\x_{t})\|^{2}}.
\end{align*}
\noindent
Consider the function $ h(x; \beta) = \sqrt{x^2 + \beta} - x$. It is easy to verify that, for $ \beta \geq 0 $, $ h(x) $ is decreasing function of $ x $. Now since 
$\lin{\nabla F(\x_{t}), \H_{t} \nabla F(\x_{t})} \leq K_{H} \|\nabla F(\x_{t})\|^{2}$, we get
\begin{align}
\label{eq:norm_cauchy_direction}
\|\s_{t}^{C}\|  = \alpha \|\nabla F(\x_{t})\| \geq \frac{1}{2\sigma_t}\left[\sqrt{K_{H}^{2} + 4 \sigma_{t} \|\nabla F(\x_{t})\|} - K_{H}\right]. 
\end{align}
Now, from \cite[Lemma 2.1]{cartis2012complexity}, we get
\begin{align*}
- m_{t}(s_{t}^{C}) \geq \frac{\sigma_{t} \|\s_{t}^{C}\|^{3}}{6}  = \frac{\|\s_{t}^{C}\|^{2} }{6}\alpha \sigma_t\|\nabla F(\x_{t})\|  \geq \frac{ \|\s_{t}^{C}\|^{2}}{12} ( \sqrt{K_{H}^{2} + 4 \sigma_{t} \|\nabla F(\x_{t})\|} - K_{H}).
\end{align*}
Alternatively, following the proof of \cite[Lemma 2.1]{cartis2011adaptiveI}, for any $ \alpha \geq 0 $, we get
\begin{align*}
m_{t}(s_{t}^{C}) &\leq m_{t}(- \alpha \nabla F(\x_{t})) \\
&= -\alpha \|\nabla F(\x_{t})\|^{2} + \frac{1}{2} \alpha^{2} \lin{\nabla F(\x_{t}), \H_{t} \nabla F(\x_{t})} + \frac{\alpha^{3}}{3}\sigma_{t} \|\nabla F(\x_{t})\|^{3}\\
& \leq \frac{\alpha \|\nabla F(\x_{t})\|^{2}}{6} \left( - 6 + 3 \alpha K_{H} + 2 \alpha^{2}\sigma_{t} \|\nabla F(\x_{t})\| \right).
\end{align*}
Consider the quadratic polynomial $r(\alpha) = 2 \alpha^{2}\sigma_{t} \|\nabla F(\x_{t})\| + 3 \alpha K_{H} - 6$.
We have $ r(\alpha) \leq 0 $ for $ \alpha \in [0,\bar{\alpha}] $, where
\begin{align*}
\bar{\alpha} = \frac{-3 K_{H} + \sqrt{9 K_{H}^{2} + 48 \sigma_{t} \|\nabla F(\x_{t})\|}}{4 \sigma_{t} \|\nabla F(\x_{t})\|}= \frac{ 12 }{\left(3 K_{H} + \sqrt{9 K_{H}^{2} + 48 \sigma_{t} \|\nabla F(\x_{t})\|}\right)}.
\end{align*}
%We can express $ \bar{\alpha} $ as
%\begin{align*}
%\bar{\alpha} =c\frac{ 12 }{\left(3 K_{H} + \sqrt{9 K_{H}^{2} + 48 \sigma_{t} \|\nabla F(\x_{t})\|}\right)}.
%\end{align*}
Note that
% \begin{align*}
$\sqrt{9 K_{H}^{2} + 48 \sigma_{t} \|\nabla F(\x_{t})\|} 
%& \leq 2 \max \left\{3 K_{H} , 4 \sqrt{3 \sigma_{t} \|\nabla F(\x_{t}) \|} \right\} 
\leq 8 \sqrt{3} \max \left\{K_{H} , \sqrt{\sigma_{t} \|\nabla F(\x_{t}) \|} \right\}$
% \end{align*}
and trivially
% \begin{align*}
$3 K_{H} \leq 4 \sqrt{3} \max \left\{K_{H} , \sqrt{\sigma_{t} \|\nabla F(\x_{t}) \|} \right\}$.
% \end{align*}
Hence, defining
%\begin{align*}
$\alpha_{0} \triangleq 1/(\sqrt{3} \max \{K_{H} , \sqrt{\sigma_{t} \|\nabla F(\x_{t}) \|} \})$, it is easy to see that 
%\end{align*}
% we have
% \begin{align*}
% \frac{3 K_{H} + \sqrt{9 K_{H}^{2} + 48 \sigma_{t} \|\nabla F(\x_{t})\|}}{12} \leq \sqrt{3} \max \left\{K_{H} , \sqrt{\sigma_{t} \|\nabla F(\x_{t}) \|} \right\},
% \end{align*}
% which implies
 $ 0 < \alpha_{0} \leq \bar{\alpha}$. With this $ \alpha_{0} $, we get $ r(\alpha_{0}) \leq 2/9 + 3/\sqrt{3} - 6 \leq -3.$ Therefore
\begin{align*}
m_{t}(s_{t}) &\leq \frac{- 3 \|\nabla F(\x_{t})\|^{2}}{6 \sqrt{3} \max \left\{K_{H} , \sqrt{\sigma_{t} \|\nabla F(\x_{t}) \|} \right\}} \\
&= \frac{-  \|\nabla F(\x_{t})\|}{2 \sqrt{3} } \min \left\{\frac{\|\nabla F(\x_{t}) \|}{K_{H}} , \sqrt{\frac{\|\nabla F(\x_{t}) \|}{\sigma_{t}}} \right\}.
\end{align*}
\qed
\end{proof}

% -------------------- % Descent with Negative Curvature % -------------------- %
\begin{lemma}[Descent with Negative Curvature]
	\label{lemma:SARC_Eigen}
	Suppose $ \lmin(\H_t) < 0 $. For some $ \nu \in (0,1] $, define
	% \begin{align*}
	$\s^{E}_{t}  = \alpha \u_{t}$, where $\alpha = \arg\min_{\widehat{\alpha} \in \mathbb{R}} m_{t}(\widehat{\alpha} \u_{t})$,
%	\end{align*}
	 and 
%	 \begin{align*}
	 $\lin{\u_{t}, \H_{t} \u_{t}} \leq \nu \lambda_{\min}(\H_{t}) \|\u_t\|^{2} < 0$.
	 % \end{align*} 
	 We have
	\begin{align*}
	- m_t(\s_t^E) \geq \frac{\nu |\lambda_{\min}(\H_{t})|}{6}  \max \left\{ \|\s_{t}^{E}\|^{2} , \frac{\nu^{2} |\lambda_{\min}(\H_{t})|^{2}}{\sigma_{t}^{2}} \right\}.
	\end{align*}
\end{lemma}
\begin{proof}
By the first-order necessary optimality condition of $ \alpha $, we get $\lin{\nabla F(\x_{t}), \u_{t}} + \alpha \lin{\u_{t}, \H_{t} \u_{t}} + \sigma_{t} \alpha^{2} \|\u_{t}\|^{3} = 0$, which implies $ \lin{\nabla F(\x_{t}), \s_{t}^{E}} + \lin{\s_{t}^{E}, \H_{t} \s_{t}^{E}} + \sigma_{t} \|\s_{t}^{E}\|^{3} = 0 $. Next, since $ \alpha $ is a minimizer of $ m_{t}(\widehat{\alpha} \u_{t}) $, we have $ m_{t}(\alpha \u_{t}) \leq m_{t}(-\alpha \u_{t}) $, which implies $ \lin{\nabla F(\x_{t}), \s_{t}^{E}} \leq 0 $. Hence, we also obtain $ \lin{\s_{t}^{E}, \H_{t} \s_{t}^{E}} + \sigma_{t} \|\s_{t}^{E}\|^{3} \geq 0$.
%\begin{align*}
%\lin{\nabla F(\x_{t}), \s_{t}^{E}} + \lin{\s_{t}^{E}, \H_{t} \s_{t}^{E}} + \sigma_{t} \|\s_{t}^{E}\|^{3} &= 0, \\
%\lin{\s_{t}^{E}, \H_{t} \s_{t}^{E}} + \sigma_{t} \|\s_{t}^{E}\|^{3} &\geq 0.
%\end{align*}
%It immediately follows that $ \lin{\nabla F(\x_{t}),\s_{t}^{E}} \leq 0 $. 
From \cite[Lemma 2.1]{cartis2012complexity}, we get 
$- m_{t}(s_{t}^{E}) \geq \sigma_{t} \|\s_t^{E}\|^{3}/6 = \left(- \lin{\nabla F(\x_{t}), \s_{t}^{E}} - \lin{\s_{t}^{E}, \H_{t} \s_{t}^{E}}\right)/6 \geq \nu |\lambda_{\min}(\H_{t})| \|\s_t^{E}\|^{2}/6$.
Now, we have
\begin{align}
\label{eq:norm_eig_direction}
\sigma_{t} \|\s_{t}^{E}\| \geq -\frac{\lin{\s_{t}^{E}, \H_{t} \s_{t}^{E}} }{\|\s_{t}^{E}\|^{2}} \geq \nu |\lambda_{\min}(\H_{t})|,
\end{align}
which gives $\sigma_{t} \|\s_{t}^{E}\|^{3} \geq \nu |\lambda_{\min}(\H_{t})| \|\s_{t}^{E}\|^{2}$ and $\sigma_{t} \|\s_{t}^{E}\|^{3} \geq \nu^{3}\sigma_{t}^{-2}|\lambda_{\min}(\H_{t})|^{3}$.
%and
%\begin{align*}
%\sigma_{t} \|\s_{t}^{E}\|^{3} \geq \frac{\nu^{3}}{\sigma_{t}^{2}}|\lambda_{\min}(\H_{t})|^{3}.
%\end{align*}
Hence, we have $ -m_{t}(\s_{t}^{E}) \geq \sigma_{t} \|\s_t^{E}\|^{3}/6 \geq \nu |\lambda_{\min}(\H_{t})| \|\s_{t}^{E}\|^{2}/6 $ and $ -m_{t}(\s_{t}^{E}) \geq  \sigma_{t} \|\s_t^{E}\|^{3}/6 \geq \nu^{3}\sigma_{t}^{-2}|\lambda_{\min}(\H_{t})|^{3}/6 $.
%\begin{align*}
%-m_{t}(\s_{t}^{E}) &\geq \frac{1}{6} \sigma_{t} \|\s_t\|^{3} \geq \frac{1}{6}\nu |\lambda_{\min}(\H_{t})| \|\s_{t}^{E}\|^{2},
%\\
%-m_{t}(\s_{t}^{E}) &\geq \frac{1}{6} \sigma_{t} \|\s_t\|^{3} \geq \frac{1}{6}\nu^{3}\sigma_{t}^{-2}|\lambda_{\min}(\H_{t})|^{3}.
%\end{align*}
%Adding the last two inequalities, we get
%\begin{align*}
%-m_{t}(\s_{t}^{E}) \geq \frac{1}{12}\nu |\lambda_{\min}(\H_{t})| \left( \|\s_{t}^{E}\|^{2} + \frac{\nu^{2}}{\sigma_{t}^{2}}|\lambda_{\min}(\H_{t})|^{2}\right).
%\end{align*}
\qed
\end{proof}

The next lemma is used to show sufficient decrease in the objective function using the approximate solution of the sub-problem \eqref{eq:SARC_subp}. 
% -------------------- % function decrease % -------------------- %
\begin{lemma}
	\label{lemma:SARC_fundecrease}
	Given Assumption \ref{assumption:a1} and Condition \ref{condition:Hessian_approximation_H}, we have
	\begin{align*}%\label{eq:SARC_fundec}
	F(\x_t + \s_t) - F(\x_t) - m_t(\s_t)  \le \left(\frac{L}{2} - \frac{\sigma_{t}}{3} \right) \|\s_t\|^3 +  \frac{\epsilon}{2}\|\s_t\|^2.
	\end{align*}
\end{lemma}
\begin{proof}
	Apply Mean Value Theorem on $F$ at $\x_t$ gives $F(\x_t + \s_t) = F(\x_t) + \nabla F(\x_t)^T \s_t + \frac{1}{2}\s_t^T \nabla^2F(\xi_t)\s_t$, for some $\xi_t$ in the segment of $[\x_t, \x_t + \s_t]$. Now, it follows that
	\begin{align*}
	&F(\x_t + \s_t) - F(\x_t) - m_t(\s_t) = \frac{1}{2} \s_t^T(\nabla^2 F(\xi_t) - \H_t)\s_t - \frac{\sigma_{t}}{3} \|\s_t\|^{3}
	\\ &= \frac{1}{2}\s_t^T(\nabla^2 F(\xi_t) - \nabla^2 F(\x_t) + \nabla^2 F(\x_t) - \H_t)\s_t - \frac{\sigma_{t}}{3} \|\s_t\|^{3}
	\\ &\le \frac{1}{2}\s_t^T(\nabla^2 F(\xi_t) - \nabla^2 F(\x_t))\s_t + \frac{1}{2}\s_t^T(\nabla^2 F(\x_t) - \H_t)\s_t - \frac{\sigma_{t}}{3} \|\s_t\|^{3}
	\\ &\le \frac{L}{2} \|\s_t\|^3 +  \frac{1}{2}\epsilon\|\s_t\|^2 - \frac{\sigma_{t}}{3} \|\s_t\|^{3} \le \left(\frac{L}{2} - \frac{\sigma_{t}}{3} \right) \|\s_t\|^3 +  \frac{\epsilon}{2}\|\s_t\|^2.
	\end{align*}
\qed
\end{proof}

% -------------------- % success % -------------------- %
\begin{lemma}
	\label{lemma:SARC_fun_best_epsilon}
	Given Assumption \ref{assumption:a1}, Conditions \ref{condition:Hessian_approximation_H} and \ref{condition:SARC_sufficient_descent}, suppose 
	\begin{align*}
	\sigma_t \geq 2 L, ~~~
	\epsilon \leq \min \left\{ \frac{1}{12} \left(\sqrt{K_{H}^{2} + 8 L \epsilon_{g}} - K_{H}\right), \frac{\nu \epsilon_{H}}{6\gamma} \right\}.
	\end{align*}
	Then, we have
	\begin{align*}
%	\label{eq:SARC_norm_s_inequality}
	\left(\frac{L}{2}- \frac{\sigma_{t}}{3} \right) \|\s_t\|^3 +  \frac{\epsilon}{2}\|\s_t\|^2 \leq \left\{
	\begin{array}{ll}
	\frac{\epsilon}{2} \|\s_{t}^{C}\|^2, \\ \\
	\frac{\epsilon}{2} \|\s_{t}^{E}\|^2, \quad \text{If} \;\; \lmin (\H_{t}) \geq - \epsilon_{H}\\
	\end{array}
	\right..
	\end{align*}
%	where $ \s_{t}^{C} $ and $ \s_{t}^{E} $ are defined as in Lemmas \ref{lemma:SARC_Cauchy} and \ref{lemma:SARC_Eigen}, respectively.
\end{lemma}
\begin{proof}
    First consider $  \|\s_{t}^{C}\| $ for which we have two cases.
	\begin{enumerate}
		\item If $ \|\s_{t}\| \leq \|\s_{t}^{C}\| $, then from assumption on $ \sigma_{t} $, it immediately follows that
		\begin{align*}
		\left(\frac{L }{2}- \frac{\sigma_{t}}{3} \right) \|\s_t\|^3 +  \frac{\epsilon}{2}\|\s_t\|^2 \leq \frac{\epsilon}{2}\|\s_t\|^2 \leq \frac{\epsilon}{2}\|\s_{t}^{C}\|^2 .
		\end{align*} 
		\item If $ \|\s_{t}\| \geq \|\s_{t}^{C}\| $, since $L \le \sigma_t/2$, then
		\begin{align*}
		\left(\frac{L }{2}- \frac{\sigma_{t}}{3} \right) \|\s_t\|^3 +  \frac{\epsilon}{2}\|\s_t\|^2 &\leq  -\frac{\sigma_t}{12}\norm{\s_t}^3 + \frac{\epsilon}{2}\norm{\s_t}^2  \le \left( -\frac{\sigma_t}{12}\norm{\s_t^C} + \frac{\epsilon}{2}\right) \norm{\s_t}^2
		\\ &\le \left( - \frac{\sqrt{K_H^2 + 8L\epsilon_g}-K_H}{24} + \frac{\epsilon}{2}\right)\norm{\s_t}^2 \le 0 \le \frac{\epsilon}{2}\norm{\s_t^C}^2.
		\end{align*}
The second last inequality follows from \eqref{eq:norm_cauchy_direction}.
    \end{enumerate}
    Similarly, for $ \|\s_{t}^{E}\| $, we have two cases.
	\begin{enumerate}
		\item If $ \|\s_{t}\| \leq \|\s_{t}^{E}\| $, then from assumption on $ \sigma_{t} $, it immediately follows that
		\begin{align*}
		\left(\frac{L}{2} - \frac{\sigma_{t}}{3} \right) \|\s_t\|^3 +  \frac{\epsilon}{2}\|\s_t\|^2 \leq \frac{\epsilon}{2}\|\s_t\|^2 \leq \frac{\epsilon}{2}\|\s_{t}^{E}\|^2 .
		\end{align*} 
		\item If $ \|\s_{t}\| \geq \|\s_{t}^{E}\| $, since $L \le \sigma_t/2$, then
		\begin{align*}
		\left(\frac{L }{2}- \frac{\sigma_{t}}{3} \right) \|\s_t\|^3 +  \frac{\epsilon}{2}\|\s_t\|^2 &\leq -\frac{\sigma_t}{12}\norm{\s_t}^3 + \frac{\epsilon}{2}\norm{\s_t}^2 \le -\frac{\sigma_t}{12}\norm{\s_t^E}\norm{\s_t}^2  + \frac{\epsilon}{2}\norm{\s_t}^2
		\\ & \le - \frac{\nu\epsilon_H}{12}\norm{\s_t}^2 + \frac{\epsilon}{2}\norm{\s_t}^2 \le 0 < \frac{\epsilon}{2}\norm{\s_t^E}^2.
		\end{align*}
The second last inequality follows from \eqref{eq:norm_eig_direction} and the last line follows from $\epsilon \le \frac{\nu\epsilon_H}{6}$. \qed
    \end{enumerate}
\end{proof}
% -------------------- % eigen direction % -------------------- %
\begin{lemma}%{lemma}{lemmodel_hess}
	\label{lemma:SARC_model_hess}
	Given Assumption \ref{assumption:a1}, Conditions \ref{condition:Hessian_approximation_H} and \ref{condition:SARC_sufficient_descent}, suppose at the t-th iteration, $\lmin(\H_t) <-\epsilon_{H}$, $\sigma_t \geq 2 L$, and $\epsilon \leq \min\{1/6, (1-\eta)/3\}\nu\epsilon_{H}$. Then, the t-th iteration is {\it successful}, i.e. $\sigma_{t+1} = \sigma_{t}/\gamma$.
\end{lemma}
\begin{proof}
	From \eqref{eq:SARC_Eigen}, Lemma \ref{lemma:SARC_fundecrease}, Lemma \ref{lemma:SARC_fun_best_epsilon}, as well as assumptions on $ \sigma_{t} $ and $\epsilon$, we have
	\begin{align*}
	1- \rho_t &= \frac{F(\x_t + \s_t) - F(\x_t) -m_t(\s_t) }{-m_t(\s_t)}  \le \frac{\left({L}/{2} - {\sigma_{t}}/{3} \right) \|\s_t\|^3 +  \epsilon\|\s_t\|^2/2}{\nu |\lambda_{\min}(\H_{t})| \|\s^{E}_{t}\|^{2}/6} \\
	& \le \frac{3\epsilon\|\s^{E}_t\|^2}{\nu |\lambda_{\min}(\H_{t})| \|\s^{E}_{t}\|^{2}} \le  \frac{3\epsilon}{\nu\epsilon_{H}}\le  1-\eta.
	\end{align*}
	Hence, $\rho_t \ge \eta$, and the iteration is successful. \qed
\end{proof}

% -------------------- % gradient direction % -------------------- %
\begin{lemma}
	\label{lemma:SARC_model_grad}
Given Assumption \ref{assumption:a1}, Conditions \ref{condition:Hessian_approximation_H} and \ref{condition:SARC_sufficient_descent}, suppose at the t-th iteration, $\|\nabla F(\x_{t})\| \geq \epsilon_{g}$, $\sigma_t \geq 2L$, and $$\epsilon \leq \min\left\{\frac{1}{12},\frac{1 - \eta}{6}\right\} \left(\sqrt{K_{H}^{2} + 8 L \epsilon_{g}} - K_{H}\right).$$
Then, the t-th iteration is {\it successful}, i.e. $\sigma_{t+1} = \sigma_{t}/\gamma$.
\end{lemma}
\begin{proof}
First note that, from \eqref{eq:SARC_Cauchy}, we have
	\begin{align*}
	-m_{t}(\s_{t}) &\geq -m_{t}(\s_{t}^{C}) \ge \frac{1}{12} \|\s_t^{C}\|^{2} \left( \sqrt{K_{H}^{2} + 4 \sigma_{t} \|\nabla F(\x_{t})\|} - K_{H}\right).
	\end{align*}
	Hence, again, by \eqref{eq:SARC_Cauchy}, Lemma \ref{lemma:SARC_fundecrease} and \ref{lemma:SARC_fun_best_epsilon}, it follows that 
	\begin{align*}
	1- \rho_t &= \frac{F(\x_t + \s_t) - F(\x_t) -m_t(\s_t) }{-m_t(\s_t)}  \leq \frac{\left(\frac{L}{2} - \frac{\sigma_{t}}{3} \right) \|\s_t\|^3 +  \frac{\epsilon}{2}\|\s_t\|^2 }{-m_t(\s_{t}^{C})} \\
	&\le \frac{\frac{\epsilon}{2}\|\s_{t}^{C}\|^2}{ \frac{1}{12}  \|\s_t^{C}\|^{2} \left( \sqrt{K_{H}^{2} + 4 \sigma_{t} \|\nabla F(\x_{t})\|} - K_{H}\right)} \le \frac{6 \epsilon}{  \left( \sqrt{K_{H}^{2} + 4 \sigma_{t} \|\nabla F(\x_{t})\|} - K_{H}\right)} \\
	&\le \frac{6 \epsilon}{ \left( \sqrt{K_{H}^{2} + 8 L \epsilon_{g}} - K_{H}\right)} \le 1-\eta.
	\end{align*}
	Hence, $\rho_t \ge \eta$, and the iteration is successful.	\qed 
\end{proof}

Now we can upper bound the cubic regularization parameter before the algorithm terminates, as in Lemma \ref{lemma:SARC_upperbound_sigma}.
% -------------------- % upper bound sigma % -------------------- %

\begin{lemma}
	\label{lemma:SARC_upperbound_sigma}
	Consider Assumption \ref{assumption:a1}, Conditions \ref{condition:Hessian_approximation_H} and \ref{condition:SARC_sufficient_descent}, and
	\begin{align}
	\label{eq:SARC_epsilon}
	\epsilon &\leq \min \left\{ \min\left\{\frac{1}{12},\frac{1-\eta}{6}\right\} \left(\sqrt{K_{H}^{2} + 8 L \epsilon_{g}} - K_{H}\right), \min\left\{\frac{1}{6},\frac{1-\eta}{3}\right\}\nu \epsilon_{H} \right\},
	\end{align}
	where $ \nu, L, K_{H}$ are, respectively, defined as in \eqref{eq:SARC_Eigen}, \eqref{eq:Hessian_Lipschitz_F}, \eqref{eq:Hessian_boundedness_H}, and $\eta $  is a hyper-parameter of Algorithm \ref{alg:SARC_fg}. 
	For Algorithm \ref{alg:SARC_fg} we have for all $ t $,
%	\begin{align}
%	\label{eq:SARC_sigma_max}
	$\sigma_{t} \le \max\{\sigma_0, 2 \gamma L\}$.
%	\end{align}
\end{lemma}	

\begin{proof}
	We prove by contradiction. Assume the t-th iteration is the first unsuccessful iteration such that $\sigma_{t+1} = \gamma \sigma_{t} \ge 2\gamma  L$,
	which implies that $\sigma_{t} \ge 2L$.	However, according to Lemmas \ref{lemma:SARC_model_hess} and \ref{lemma:SARC_model_grad}, respectively, if $\lambda_{\min}(H_{t}) < -\epsilon_{H}$ or $ \|\nabla F(\x_{t})\| \geq \epsilon_{g}$, then the iteration is successful and hence we must have $ \sigma_{t+1} = \sigma_{t}/\gamma \leq \sigma_{t}$, which is a contradiction. 
\qed
\end{proof}

Now, similar to \cite[Lemma 2.8]{cartis2012complexity}, we can get the following result about the estimate of the total number of successful iterations before algorithm terminates.
\begin{lemma}[Success Iterations]
	\label{lemma:SARC_succ}
	Given Assumption \ref{assumption:a1}, Conditions \ref{condition:Hessian_approximation_H} and \ref{condition:SARC_sufficient_descent}, let $\mathcal{T}_{succ}$ denote the set of all the successful iterations before Algorithm \ref{alg:SARC_fg} stops. The number of successful iterations is upper bounded by,
	\begin{align*}
	\Abs{\mathcal{T}_{succ}} \le \frac{(F(\x_0) - F_{\min})}{\eta \kappa_{\sigma}} \cdot \max\{\epsilon_{g}^{-2}, \epsilon_{H}^{-3}\},
	\end{align*}
    where $\kappa_{\sigma}  \triangleq \min\left\{\nu^{3}/(24 \gamma^{2} L^{2}) , \min \left\{{1}/{K_{H}} , \sqrt{{1}/{(2 \gamma L)}} \right\}/(2 \sqrt{3}) \right\}$.
\end{lemma}
\begin{proof}
	Suppose Algorithm \ref{alg:SARC_fg} doesn't terminate at the t-th iteration. Then either we have $\norm{\nabla F(\x_{t})} \ge \epsilon_g$ or $\lmin(\nabla^2 \H_{t}) \le -\epsilon_{H}$. In the first case, \eqref{eq:SARC_Cauchy} and Lemma \ref{lemma:SARC_upperbound_sigma} gives
	\begin{align*}
	-m_t(\s_t) &\ge \frac{\|\nabla F(\x_{t})\|}{2 \sqrt{3} } \min \left\{\frac{\|\nabla F(\x_{t}) \|}{K_{H}} , \sqrt{\frac{\|\nabla F(\x_{t}) \|}{\sigma_{t}}} \right\} 
	%\ge \frac{4 \epsilon_{g}}{6 \sqrt{3} } \min \left\{\frac{\epsilon_{g}}{K_{H}} , \sqrt{\frac{\epsilon_{g}}{2 \gamma L}} \right\} 
	\ge \frac{ \epsilon_{g}^{2}}{2 \sqrt{3} } \min \left\{\frac{1}{K_{H}} , \sqrt{\frac{1}{2 \gamma L}} \right\}.
	\end{align*}
	Similarly, in the case where $\lmin(\nabla^2 \H_{t}) \le -\epsilon_{H}$, from \eqref{eq:SARC_Eigen} and Lemma  \ref{lemma:SARC_upperbound_sigma}, we obtain 
	$-m_{t}(\s_{t}) \geq \nu^{3}|\lambda_{\min}(\H_{t})|^{3}/(6 \sigma_{t}^{2}) \geq \nu^{3} \epsilon_{H}^{3}/(24 \gamma^{2} L^{2})$.	
	
	Since $F(\x_t)$ is monotonically decreasing, we have
	\begin{align*}
	& F(\x_0) - F_{\min} \ge \sum_{t=0}^\infty F(\x_t) - F(\x_{t+1}) \ge \sum_{t\in \mathcal{T}_{succ}} F(\x_t) - F(\x_{t+1})	
	\ge - \eta \sum_{t\in \mathcal{T}_{succ}} m_{t}(\s_{t}) 
	\\ &\ge \eta \Abs{\mathcal{T}_{succ}} \min\left\{\frac{\nu^{3} \epsilon_{H}^{3}}{24 \gamma^{2} L^{2}} ,\frac{ \epsilon_{g}^{2}}{2 \sqrt{3} } \min \left\{\frac{1}{K_{H}} , \sqrt{\frac{1}{2 \gamma L}} \right\} \right\} 
	\ge \Abs{\mathcal{T}_{succ}} \eta \kappa_{\sigma} \min\{\epsilon_g^{2}, \epsilon_H^{3}\}. \; \qed
	\end{align*}
%	Therefore we have
%	$$\Abs{\mathcal{T}_{succ}} \le \frac{(F(\x_0) - F_{\min})}{\eta \kappa_{\sigma}} \cdot \max\{\epsilon_{g}^{-2}, \epsilon_{H}^{-3}\}.$$ \qed
\end{proof}

Now we show the final complexity bounds of Algorithm \ref{alg:SARC_fg} in Theorem \ref{theorem:SARC_main_det}.
\begin{theorem}[Complexity of Algorithm \ref{alg:SARC_fg}]%{theorem}{theoremSARC}
	\label{theorem:SARC_main_det}
	Consider any $ 0< \epsilon_{g},\epsilon_{H} < 1 $. Suppose the inexact Hessian, $ \H(\x)$, satisfies Condition \ref{condition:Hessian_approximation_H} with the approximation tolerance, $ \epsilon $, in \eqref{eq:Hessian_accuracy_H} as \eqref{eq:SARC_epsilon}.
%	\begin{align}
%	\label{eq:SARC_epsilon}
%	\epsilon \leq \min \left\{ \min\left\{\frac{1}{12},\frac{1-\eta}{6}\right\} \left(\sqrt{K_{H}^{2} + 8 L \epsilon_{g}} - K_{H}\right), \min\left\{\frac{1}{6},\frac{1-\eta}{3}\right\}\nu \epsilon_{H} \right\},
%	\end{align}
%	 where $ \nu, L, K_{H}$ are, respectively, defined as in \eqref{eq:SARC_Eigen}, \eqref{eq:Hessian_Lipschitz_F}, \eqref{eq:Hessian_boundedness_H}, and $\gamma, \eta $  are the hyper-parameters of Algorithm \ref{alg:SARC_fg}. 
	 For Problem \eqref{eq:obj}, under Assumption \ref{assumption:a1} and Condition \ref{condition:SARC_sufficient_descent}, 
	 %if the approximate solution to the sub-problem \eqref{eq:SARC_subp} satisfies Condition \ref{condition:SARC_sufficient_descent}, 
	  Algorithm \ref{alg:SARC_fg} terminates after  at most 
%	 \begin{align*}
	 $T \in \mathcal O\left(\max\{\epsilon_g^{-2}, \epsilon_H^{-3}\}\right)$
%	 \end{align*}
iterations.
\end{theorem}
\begin{proof}
	Suppose Algorithm \ref{alg:SARC_fg} terminates at the t-th iteration.
	Let $\mathcal{T}_{succ}$ and $\mathcal{T}_{fail}$ denote the sets of all the successful and unsuccessful iterations, respectively. Then $T = \Abs{\mathcal{T}_{succ}} + \Abs{\mathcal{T}_{fail}}$ and $\sigma_T = \sigma_0 \gamma^{\Abs{\mathcal{T}_{fail}} - \Abs{\mathcal{T}_{succ}}}$. From Lemma \ref{lemma:SARC_upperbound_sigma}, we have $\sigma_T \le 2 \gamma L$. Hence,
	$\Abs{\mathcal{T}_{fail}} \leq \log\left({2 \gamma L}/{\sigma_0}\right)/\log \gamma + \Abs{\mathcal{T}_{succ}}$,
	which, using Lemma \ref{lemma:SARC_succ} gives the total iteration complexity as
	\begin{align*}
		T \le \log\left({2 \gamma L}/{\sigma_0}\right)/\log \gamma + 2(F(\x_0) - F_{\min}) \cdot \max\{\epsilon_{g}^{-2}, \epsilon_{H}^{-3}\}/(\eta \kappa_{\sigma}),
	\end{align*}
	where $ \kappa_\sigma$ is defined in Lemma \ref{lemma:SARC_succ}.
\qed
\end{proof}

In Theorem \ref{theorem:SARC_main_det} (as well as Theorem \ref{theorem:SARC_main_det_optimal} below), we require $ \epsilon \in \mathcal{O}(\sqrt{\epsilon_{g}},\epsilon_{H}) $. This can be rather strict and computationally unattractive, unless either crude solutions are required (e.g., in most machine learning applications very rough solutions are encouraged to avoid over-fitting), or the inexact Hessian is formed from a sub-set of data that is significantly smaller than the original dataset (e.g., see Section \ref{sec:sub_sampling_and_finite_sum} in the context of big-data regimes where $ n \gg 1 $ and $ |\mathcal{S}| \ll n $). Nonetheless, the theoretical existence of such tolerance, though small, implies a certain level of robustness of the algorithm, i.e., the complexity of the algorithm is not adversely affected by small errors in Hessian computations.

We note that, for iterations where $ \epsilon \ll \|\s_{t}\| $, \eqref{eq:Hessian_accuracy_H} is indeed a more stringent condition than \eqref{eq:Hessian_accuracy_H_quadratic_SARC}. As iterations progress towards optimality, step-size can become small, in which case  \eqref{eq:Hessian_accuracy_H}  might be theoretically more preferable. Nonetheless, beyond a direct theoretical comparison among various Hessian approximation bounds in terms of their tightness, the main advantage of  \eqref{eq:Hessian_accuracy_H} should be regarded in light of its simplicity, which allows for direct constructions of $ \H_t $ with a priori guarantees.

Condition \ref{condition:SARC_sufficient_descent} seems to be the bare minimum required to guarantee convergence to an approximate second-order criticality. Intuitively, however, if an approximate solution to the sub-problem \eqref{eq:SARC_subp} satisfies more than \eqref{eq:SARC_cond}, i.e., if we solve \eqref{eq:SARC_subp} more exactly than just requiring \eqref{eq:SARC_cond}, one could expect to be able to improve upon the iteration complexity of Theorem \ref{theorem:SARC_main_det}. Indeed, suppose we solve the reduced sub-problem on progressively embedded sub-spaces with increasingly higher dimensions, all of which including ``$ \text{Span}\{\s_t^C, \s_t^E\} $'', and stop when the corresponding solution $ \s_{t} $ satisfies the following conditions.

\begin{condition}[Sufficient Descent for Optimal Complexity]
	\label{condition:SARC_sufficient_descent_strict} 
	Assume that we solve the sub-problem \eqref{eq:SARC_subp} approximately to find $\s_{t}$ such that, in addition to \eqref{eq:SARC_cond}, we have
	\begin{align}
	\label{eq:SARC_Optimal_Stopping_Criterion}
%	\|\nabla m_{t}(\s_{t})\| \leq \theta_{t} \|\nabla F(\x_{t})\|, \quad \theta_{t} \triangleq \zeta \min\left\{ 1 , \|\s_{t}\| \right\},
	\|\nabla m_{t}(\s_{t})\| \leq \zeta \max\left\{\|\s_{t}\|^{2}, \theta_{t} \|\nabla F(\x_{t})\|\right\} , \quad \theta_{t} \triangleq \min\left\{ 1 , \|\s_{t}\| \right\},
	\end{align}
	for some prescribed $ \zeta \in (0,1) $. 
	Here, $m_t(\cdot)$ is defined in \eqref{eq:SARC_subp}.
\end{condition}

Conditions on the inexactness of the sub-problems were initially pioneered in \cite{cartis2011adaptiveI,cartis2011adaptiveII,cartis2012complexity}. However, the main drawback for these conditions is that the inexactness tolerance is closely tied with the magnitude of the gradient. More specifically, when gradient is small, e.g., near saddle points, the sub-problems are required to be solved exceedingly more accurately. In fact, at a saddle point where $ \|\nabla F(\x_{t})\| = 0 $, these conditions imply an exact solution to the sub-problem. 
To the best of our knowledge, Condition \ref{condition:SARC_sufficient_descent_strict} represents a novel criterion, which offers the best of both worlds: when gradient is large, we allow for crude solutions to the sub-problem, but near saddle-points where the gradient is small, inexactness will be determined by the step length, which can be significantly larger than the gradient.
Using Condition \ref{condition:SARC_sufficient_descent_strict}, we can obtain the optimal iteration complexity for Algorithm \ref{alg:SARC_fg}, as shown in Theorem \ref{theorem:SARC_main_det_optimal}. First, we prove the following two lemmas which will be used later for the proof of Theorem \ref{theorem:SARC_main_det_optimal}.
% The proof of Theorem \ref{theorem:SARC_main_det_optimal} is given in Appendix \ref{sec:SARC_optimal_complexity_proof}.

% -------------------- % optimal gradient % -------------------- %
\begin{lemma}
	\label{lemma:SARC_optimal_gradient}
	Suppose $ \|\nabla F(\x_{t})\| \geq \epsilon_{g} $. Given Assumption \ref{assumption:a1} and Condition \ref{condition:SARC_sufficient_descent}, let \eqref{eq:Hessian_accuracy_H} hold with $ \epsilon_{t} = \min\{\epsilon, \zeta \|\nabla F(\x_{t})\|\} $ where $ \epsilon $ is as in \eqref{eq:SARC_epsilon} and $\zeta \in (0,1/2) $. Furthermore, suppose \eqref{eq:SARC_subp} is solved such that Condition \ref{condition:SARC_sufficient_descent_strict} eventually holds. Then, we have
%	\begin{align*}
%	\label{eq:SARC_optimal_step_lowerbound}
	$\|\s_{t}\| \geq \kappa_{g} \sqrt{\|\nabla F(\x_{t+1})\|}$,
%	\end{align*}
	where
	\begin{align*}
		\kappa_{g} \triangleq \frac{2 (1- 2 \zeta)}{\left((1+4 \gamma) L + 2 \max\left\{ (\epsilon + \zeta \max\{1,K\} ), 2 \zeta \max\{1,K\}\right\}\right)}.
	\end{align*}
\end{lemma}
\begin{proof}
First, suppose $ \|\s_{t}\|^{2} \leq \theta_{t} \|\nabla F(\x_{t})\| $. Using Condition \ref{condition:SARC_sufficient_descent_strict}, we get
$\|\nabla F(\x_{t+1})\| \leq \|\nabla F(\x_{t+1}) - \nabla m_{t}(\s_{t})\| + \|\nabla m_{t}(\s_{t})\| \leq \|\nabla F(\x_{t+1}) - \nabla m_{t}(\s_{t})\| + \theta_{t} \|\nabla F(\x_{t})\|$. Noting that  $\nabla m_{t}(\s_{t}) = \nabla F(\x_{t}) + \H_{t} \s_{t} + \sigma_{t} \|\s_{t}\| \s_{t}$, and using Mean Value Theorem for vector-valued functions, \eqref{eq:Hessian_Lipschitz_F} and \eqref{eq:Hessian_accuracy_H}, we get
\begin{align*}
&\|\nabla F(\x_{t+1}) - \nabla m_{t}(\s_{t})\| \leq \|\int_{0}^{1} \nabla^{2} F(\x_{t}+\tau \s_{t}) \s_{t} d \tau - \H_{t}\s_{t}\| + \sigma_{t} \|\s_{t}\|^{2} \\
&\leq \|\int_{0}^{1} \left(\nabla^{2} F(\x_{t}+\tau \s_{t}) - \nabla^{2} F(\x_{t})\right) \s_{t} d \tau + \left(\nabla^{2} F(\x_{t}) - \H_{t} \right) \s_{t}\| + \sigma_{t} \|\s_{t}\|^{2} \\
&\leq \|\s_{t}\|\int_{0}^{1} \|\nabla^{2} F(\x_{t}+\tau \s_{t}) - \nabla^{2} F(\x_{t})\| d \tau + \| \left(\nabla^{2} F(\x_{t}) - \H_{t} \right) \s_{t}\| + \sigma_{t} \|\s_{t}\|^{2} \\
&\leq L \|\s_{t} \|^{2}\int_{0}^{1} \tau d \tau + \epsilon_{t} \|\s_{t}\| + \sigma_{t} \|\s_{t}\|^{2} 
%= \left(\frac{L}{2} +\sigma_{t} \right) \|\s_{t} \|^{2} + \epsilon_{t} \|\s_{t}\| 
\leq \left(\frac{L}{2} + 2 \gamma L \right) \|\s_{t} \|^{2} + \epsilon_{t} \|\s_{t}\|,
\end{align*}
where the last equality follows from Lemma \ref{lemma:SARC_upperbound_sigma}. From \eqref{eq:Hessian_boundedness_F}, it follows that
\begin{align}
\label{eq:grad_norm}
\|\nabla F(\x_{t})\| \leq K \|\s_{t}\| + \|\nabla F(\x_{t+1})\|.
\end{align}
As such, using $ \theta_{t} \leq \zeta $ from Condition \ref{condition:SARC_sufficient_descent_strict} as well as the assumption on $ \epsilon_{t} $, we get
\begin{align*}
\|\nabla F(\x_{t+1})\| &\leq \left(\frac{L}{2} + 2 \gamma L \right) \|\s_{t} \|^{2} + \epsilon_{t} \|\s_{t}\| + \theta_{t}  K \|\s_{t}\| + \theta_{t} \|\nabla F(\x_{t+1})\| \\
&\leq \left(\frac{L}{2} + 2 \gamma L \right) \|\s_{t} \|^{2} + \epsilon_{t} \|\s_{t}\| + \theta_{t}  K \|\s_{t}\| + \zeta \|\nabla F(\x_{t+1})\|,
\end{align*} 
which implies that $(1-\zeta) \|\nabla F(\x_{t+1})\| \leq \left(L/2 + 2 \gamma L \right) \|\s_{t} \|^{2} + \left( \epsilon_{t} + \theta_{t}  K \right) \|\s_{t}\|$.
Now using Condition  \ref{condition:SARC_sufficient_descent_strict}, we consider two cases: 
\begin{enumerate}
	\item If $ \|\s_{t}\| \geq 1$, then we get $\left( \epsilon_{t} + \theta_{t}  K \right) \|\s_{t}\| \leq \left( \epsilon_{t} + \theta_{t}  K \right) \|\s_{t}\|^{2} \leq (\epsilon + \zeta K )\|\s_{t}\|^{2}$. Hence, it follows that $(1-\zeta) \|\nabla F(\x_{t+1})\| \leq \left(L/2 + 2 \gamma L + (\epsilon + \zeta K ) \right) \|\s_{t} \|^{2}$.
	
	\item If $ \|\s_{t}\| \leq 1$, then from assumption on $ \epsilon_{t} $ and \eqref{eq:grad_norm} , we have $\epsilon_{t} \|\s_{t}\| \leq \zeta \|\nabla F(\x_{t}) \|\|\s_{t}\| \leq \zeta (K \| \s_{t} \|^{2} + \|\nabla F(\x_{t+1})\| \| \s_{t} \|) 
	\leq \zeta (K \| \s_{t} \|^{2} + \|\nabla F(\x_{t+1})\|)$.	Now by assumption on $ \theta_{t} $, we get $\left( \epsilon_{t} + \theta_{t}  K \right) \|\s_{t}\| =  \epsilon_{t}\|\s_{t}\| + \theta_{t}  K \|\s_{t}\| \leq 2 \zeta K \| \s_{t} \|^{2} + \zeta \|\nabla F(\x_{t+1})\|$, which, in turn, implies that $(1-2\zeta) \|\nabla F(\x_{t+1})\| \leq \left(L/2 + 2 \gamma L + 2 \zeta K \right) \|\s_{t} \|^{2}$.
\end{enumerate}
Now suppose, $ \|\s_{t}\|^{2} \geq \theta_{t} \|\nabla F(\x_{t})\| $. As above, we have $\|\nabla F(\x_{t+1})\| \leq \|\nabla F(\x_{t+1}) - \nabla m_{t}(\s_{t})\| + \|\nabla m_{t}(\s_{t})\| \leq \left({L}/{2} + 2 \gamma L + \zeta  \right) \|\s_{t} \|^{2} + \epsilon_{t} \|\s_{t}\|$. If $ \|\s_{t}\| \geq 1$, we have $\epsilon_{t} \|\s_{t}\| \leq \epsilon\|\s_{t}\|^{2}$, which gives $\|\nabla F(\x_{t+1})\| \leq \left({L}/{2} + 2 \gamma L + \zeta +\epsilon \right) \|\s_{t} \|^{2}$. Otherwise, if $ \|\s_{t}\| \leq 1 $, then $ \|\s_{t}\|^{2} \geq \theta_{t} \|\nabla F(\x_{t})\| $ implies that $ \|\s_{t}\| \geq \|\nabla F(\x_{t})\|$. From assumption on $ \epsilon_{t} $, it follows  that $\epsilon_{t} \|\s_{t}\| \leq \zeta \|\nabla F(\x_{t}) \|\|\s_{t}\| \leq \zeta \|\s_{t}\|^{2}$, which in turn gives $\|\nabla F(\x_{t+1})\| \leq \left({L}/{2} + 2 \gamma L + 2\zeta  \right) \|\s_{t} \|^{2}$.
\qed
%	Hence, we get the desired result.\qed
\end{proof}

% -------------------- % success iteration % -------------------- %

\begin{lemma}[Success Iterations: Optimal Case]
	\label{lemma:SARC_succ_optimal}
	Let $$\mathcal{T}_{\text{succ}} \triangleq \{t; \; \|\nabla F(\x_{t})\| \geq \epsilon_{g} \; \lor \; \lmin(\H_{t}) \leq -\epsilon_{H} \},$$ be the set of all successful iterations, before Algorithm \ref{alg:SARC_fg} terminates.
	Under the conditions of Lemma \ref{lemma:SARC_optimal_gradient}, we must have
	$\Abs{\mathcal{T}_{\text{succ}}} \in \mathcal{O}(\max\{\epsilon_{H}^{-3}, \epsilon_{g}^{-3/2}\})$.
%	, where $ \eta $ is the hyper parameter of Algorithm \ref{alg:SARC_fg}, $ \sigma_{t} \geq \sigma_{\min} $ and $ \kappa_{g} $ is defined in Lemma \ref{lemma:SARC_optimal_gradient}.
\end{lemma}
\begin{proof}
	From \eqref{eq:SARC_Eigen} and Lemma \ref{lemma:SARC_upperbound_sigma}, if $\lmin(\nabla^2 \H_{t}) \le -\epsilon_{H}$, it follows that 
	$-m_{t}(\s_{t}) \geq \nu^{3}|\lambda_{\min}(\H_{t})|^{3}/(6 \sigma_{t}^{2}) \geq \nu^{3} \epsilon_{H}^{3}/(24 \gamma^{2} L^{2})$.
	Note that
$\mathcal{T}_{\text{succ}} = \mathcal{T}^{1}_{\text{succ}} \bigcup \mathcal{T}^{2}_{\text{succ}} \bigcup \mathcal{T}^{3}_{\text{succ}},
$where
	\begin{align*}
	\mathcal{T}^{1}_{\text{succ}} &\triangleq \left\{t \in \mathcal{T}_{\text{succ}}; \; \|\nabla F(\x_{t+1})\| \geq \epsilon_{g} \right\}, \\
	\mathcal{T}^{2}_{\text{succ}} &\triangleq \left\{t \in \mathcal{T}_{\text{succ}}; \; \|\nabla F(\x_{t+1})\| \leq \epsilon_{g} ~ \text{ and } ~ \lmin(H_{t+1}) \leq -\epsilon_{H} \right\} \\
		\mathcal{T}^{3}_{\text{succ}} &\triangleq \left\{t \in \mathcal{T}_{\text{succ}}; \; \|\nabla F(\x_{t+1})\| \leq \epsilon_{g} ~ \text{ and } ~ \lmin(H_{t+1}) \geq -\epsilon_{H} \right\}.
	\end{align*}
	We bound each of these sets individually. Since $F(\x_t)$ is monotonically decreasing, from \cite[Lemma 3.3]{cartis2011adaptiveI}, $ \sigma_{t} \geq \sigma_{\min} $, and Lemmas \ref{lemma:SARC_upperbound_sigma} and \ref{lemma:SARC_optimal_gradient}, we have
	\begin{align*}
	&F(\x_0) - F_{\min} \ge \sum_{t=0}^\infty F(\x_t) - F(\x_{t+1})
	 \ge \sum_{t\in \mathcal{T}^{1}_{\text{succ}}} F(\x_t) - F(\x_{t+1}) \geq - \eta \sum_{t\in \mathcal{T}^{1}_{\text{succ}}} m_{t}(\s_{t}) \\
	&\geq \eta  \sum_{t\in \mathcal{T}^{1}_{\text{succ}}} \min \left\{\frac{\nu^{3} \epsilon_{H}^{3}}{24 \gamma^{2} L^{2}},  \frac{\sigma_{\min}}{6} \|\s_{t}\|^{3}\right\} 
	\geq \eta  \sum_{t\in \mathcal{T}^{1}_{\text{succ}}} \min \left\{\frac{\nu^{3} \epsilon_{H}^{3}}{24 \gamma^{2} L^{2}},  \frac{\sigma_{\min} \kappa_{g}^{3}}{6} \|\nabla F(\x_{t+1})\|^{3/2}\right\} \\
	&\geq \eta \sum_{t\in \mathcal{T}^{1}_{\text{succ}}} \min \left\{\frac{\nu^{3} \epsilon_{H}^{3}}{24 \gamma^{2} L^{2}},  \frac{\sigma_{\min} \kappa_{g}^{3}}{6} \epsilon_{g}^{3/2}\right\} 
	\geq \eta  \sum_{t\in \mathcal{T}^{1}_{\text{succ}}} \min \left\{\frac{\nu^{3}}{24 \gamma^{2} L^{2}},  \frac{\sigma_{\min} \kappa_{g}^{3}}{6} \right\} \min\{\epsilon_{H}^{3}, \epsilon_{g}^{3/2}\}.
	\end{align*}
	Hence, $\Abs{\mathcal{T}^{1}_{\text{succ}}} \leq \kappa^{1}_{\mathcal{T}_{\text{succ}}} \max\{\epsilon_{H}^{-3}, \epsilon_{g}^{-3/2}\}$, where $$\kappa^{1}_{\mathcal{T}_{\text{succ}}} \triangleq (F(\x_0) - F_{\min}) \max \{{24 \gamma^{2} L^{2}}/{\nu^{3}},  {6}/({\sigma_{\min} \kappa_{g}^{3}}) \}/\eta.$$
	% \end{align*}
	As for $\mathcal{T}^{2}_{\text{succ}}$, we have
	\begin{align*}
	& F(\x_0) - F_{\min} \ge  F(\x_{0}) - F(\x_{1}) + \sum_{t=0}^\infty F(\x_{t+1}) - F(\x_{t+2}) \\
	& \ge F(\x_{0}) - F(\x_{1}) + \sum_{t\in \mathcal{T}^{2}_{\text{succ}}} F(\x_{t+1}) - F(\x_{t+2}) 
	\geq F(\x_{0}) - F(\x_{1}) - \eta \sum_{t\in \mathcal{T}^{2}_{\text{succ}}} m_{t+1}(\s_{t+1}) \\
	&\geq F(\x_{0}) - F(\x_{1}) + \eta  \sum_{t\in \mathcal{T}^{2}_{\text{succ}}} \frac{\nu^{3} \epsilon_{H}^{3}}{24 \gamma^{2} L^{2}}. 
	\end{align*}
	Hence,
	$\Abs{\mathcal{T}^{2}_{\text{succ}}} \leq 	\kappa^{2}_{\mathcal{T}_{\text{succ}}} \epsilon_{H}^{-3}$, where $\kappa^{2}_{\mathcal{T}_{\text{succ}}} \triangleq (F(\x_1) - F_{\min}) {24 \gamma^{2} L^{2}}/(\eta \nu^{3})$.
	Finally, we have $ \Abs{	\mathcal{T}^{3}_{\text{succ}}} = 1 $, because in such a case, the algorithm stops in one iteration. Putting these bounds all together, we get $\Abs{\mathcal{T}_{\text{succ}}} \leq \max\{1, \kappa^{1}_{\mathcal{T}_{\text{succ}}}, \kappa^{2}_{\mathcal{T}_{\text{succ}}}\} \max\{\epsilon_{H}^{-3}, \epsilon_{g}^{-3/2}\}$. \qed
\end{proof}

Now we can obtain the optimal complexity bound of Algorithm \ref{alg:SARC_fg} in Theorem \ref{theorem:SARC_main_det_optimal}. The proof follows similarly as that of Theorem \ref{theorem:SARC_main_det}, and hence is omitted here.
% -------------------- % optimal complexity bound % -------------------- %

\begin{theorem}[Optimal Complexity of Algorithm \ref{alg:SARC_fg}]
	\label{theorem:SARC_main_det_optimal}
	Consider any $ 0< \epsilon_{g},\epsilon_{H} < 1 $. Suppose the inexact Hessian, $ \H(\x)$, satisfies Conditions \eqref{eq:Hessian_approximation} with the approximation tolerance, $ \epsilon $, in \eqref{eq:Hessian_accuracy_H} as $ \epsilon = \min\{\epsilon_{0}, \zeta \epsilon_{g} \} $ where $ \epsilon_{0} $ is as in \eqref{eq:SARC_epsilon}, and $\zeta \in (0,1/2) $. 
	For Problem \eqref{eq:obj} and under Assumption \ref{assumption:a1}, if the approximate solution to the sub-problem \eqref{eq:SARC_subp} satisfies Conditions \ref{condition:SARC_sufficient_descent} and \ref{condition:SARC_sufficient_descent_strict}, then Algorithm \ref{alg:SARC_fg} terminates after  at most 
%	\begin{align*}
	$T \in \bigO\left(\max\{\epsilon_g^{-3/2}, \epsilon_H^{-3}\}\right)$
%	\end{align*}
	iterations.
\end{theorem}

From \eqref{eq:Hessian_accuracy_H}, upon termination of Algorithm \ref{alg:SARC_fg}, the obtained solution satisfies $(\epsilon_g, \epsilon + \epsilon_H)$-Optimality as in \eqref{eq:2ndopt}, i.e., $ \norm{\nabla F(\x_{T})} \leq \epsilon_{g} $ and $ \lambda_{\min}\left(\nabla^{2} F(\x_{T})\right) \geq -(\epsilon_{H} + \epsilon) $.

%Note that in order to get an accumulative success probability of $1-\delta_{0}$ for $T$ iterations, the per-iteration failure probability is set as $\delta = 1- \sqrt[T]{(1-\delta_{0})}\in \mathcal{O}(\delta_{0}/T)$. 
%In our analysis, the failure probability appears in the ``$ \log $'' term. Hence, one can choose a fairly small failure probability and ensure that even after many iterations, with high probability, all iterations are successful. More precisely, 

\section{Finite-Sum Minimization}
\label{sec:sub_sampling_and_finite_sum}
In this section, we give concrete and practical examples to demonstrate ways to construct the approximate Hessian, which satisfies Condition \ref{condition:Hessian_approximation_H}. By considering \emph{finite-sum minimization}, a ubiquitous problem arising frequently in machine learning, we showcase the practical benefits of the proposed relaxed requirement \eqref{eq:Hessian_accuracy_H} for approximating Hessian, compared to the stronger alternative \eqref{eq:Hessian_accuracy_H_quadratic_SARC}. In Section \ref{sec:randomzied_sub-sampling}, we describe randomized techniques to appropriately construct the approximate Hessian, followed by the convergence analysis of Algorithms \ref{alg:STR_fg} and \ref{alg:SARC_fg} with such Hessian approximations in Section \ref{sec:probabilistic_convergence}.

\subsection{Randomized Sub-Sampling}
\label{sec:randomzied_sub-sampling}
Indeed, a major advantage of \eqref{eq:Hessian_accuracy_H} over \eqref{eq:Hessian_accuracy_H_quadratic_SARC} is that there are many approximation techniques that can produce an inexact Hessian satisfying \eqref{eq:Hessian_accuracy_H}. Of particular interest in our present paper is the application of randomized matrix approximation techniques, which have recently shown great success in the area of RandNLA at solving various numerical linear algebra tasks \cite{woodruff2014sketching,mahoney2011randomized,drineas2016randnla}. For this, we consider the highly prevalent finite-sum minimization problem \eqref{eq:obj_sum} and employ random sampling as a way to construct approximations to the exact Hessian, which are, probabilistically, ensured to satisfy \eqref{eq:Hessian_accuracy_H}. 
%More specifically, in this section, we consider the optimization problem \eqref{eq:obj_sum}
%\begin{align}
%%\label{eq:obj_sum}
%\min_{\x \in \bbR^d} F(\x) \triangleq \frac{1}{n}\sum_{i=1}^n f_i(\x), \tag{{\bf P1}}
%\end{align}
%where each $f_i(\x)$ is a smooth but possibly non-convex function. 
Many machine learning and scientific computing applications involve finite-sum optimization problems of the form \eqref{eq:obj_sum} where each $f_{i}$ is a loss (or misfit) function corresponding to $i^{th}$ observation (or measurement), e.g., see \cite{rodoas1,rodoas2,roszas,friedman2001elements,bottou2016optimization,sra2012optimization} and references therein. 
%In particular, in machine learning applications, $F$ in \eqref{eq:obj_sum} corresponds to the \emph{empirical risk} \cite{shalev2014understanding} and the goal of solving \eqref{eq:obj} is to obtain a solution with small generalization error, i.e., high predictive accuracy on ``unseen'' data. 

Here, we consider \eqref{eq:obj_sum} in large-scale regime where $n , d \gg 1$. In such settings, the mere evaluations of the Hessian and the gradient increase linearly in $ n $. Indeed, for big-data problems, the operations with the Hessian, e.g., matrix-vector products involved in the (approximate) solution of the sub-problems \eqref{eq:STR_subp} and \eqref{eq:SARC_subp}, typically constitute the main bottleneck of computations, and in particular when $ n \gg 1 $, are computationally prohibitive. 
%One can argue that this challenge alone has greatly contributed in the fall of many second order methods from the machine learning community's radar. 
%In fact, the cost of evaluating the exact gradient is usually amortized by that of operations with the Hessian. 
For the special case of \eqref{eq:obj_sum} in which each $ f_{i} $ is convex, randomized sub-sampling has shown to be effective in reducing such costs, e.g., \cite{roosta2018sub,xu2016sub,bollapragada2016exact}. We now show that such randomized approximation techniques can indeed be effectively employed for the non-convex settings considered in this paper. 

In this light, suppose we have a probability distribution, $ \p = \{p_{i}\}_{i=1}^{n} $, over the set $\{1,2,\ldots,n\}$, such that for each index $ i =1,2\ldots,n$, we have $ \Pr(i) = p_{i} > 0 $ and $ \sum_{i=1}^{n} p_{i} = 1$. Consider picking a sample of indices from $\{1,2,\ldots,n\}$, at each iteration, randomly according to the distribution $ \p $. Let $\mathcal{S}$ and $|\mathcal{S}|$ denote the sample collection and its cardinality, respectively and define 
\begin{equation}
%\H(\x) \triangleq \frac{1}{|\mathcal{S}|} \sum_{j \in \mathcal{S}} \nabla^{2} f_{j}(\x),
\H(\x) \triangleq \frac{1}{n |\mathcal{S}| } \sum_{j \in \mathcal{S}} \frac{1}{p_{j}}\nabla^{2} f_{j}(\x),
\label{eq:subsampled_H}
\end{equation}
to be the sub-sampled Hessian.  In big-data regime when $ n \gg 1$, if $ |\mathcal S| \ll n $, such sub-sampling can offer significant computational savings. 

Now, suppose
\begin{subequations}
	\label{eq:Hessian_boundedness_finite_sum}
	\begin{align}
	\label{eq:Hessian_boundedness_finite_sum_fi}
	\sup_{\x \in \mathbb{R}^{d}} \|\nabla^2 f_i(\x)\| \le K_{i}, \quad i = 1,2,\ldots, n,
	\end{align} 
	and define
	\begin{align}
	\label{eq:Hessian_boundedness_finite_sum_K}
	K_{\max} &\triangleq \max_{i=1,\ldots,n} K_{i}.\\
    \widehat K &\triangleq \frac{1}{n}\sum_{i=1}^n K_i.\label{eq:Hessian_boundedness_finite_sum_hat_K}
    \end{align}
\end{subequations}
In this case, we can naturally consider uniform distribution over $\{1,2,\ldots,n\}$, i.e., $ p_{i} = 1/n,; \forall i $. Lemma \ref{lemma:uniform} gives the sample size required for the inexact Hessian, $ \H(\x) $, to probabilistically satisfy \eqref{eq:Hessian_approximation}, for when the indices are picked uniformly at random \textit{with} or \textit{without} replacement. 
%The proof of Lemma \ref{lemma:uniform} is in Appendix \ref{sec:sampling_proofs}.

%-------------------------------%-------------------------------
%-------------------------------Lemma: uniform
\begin{lemma}[Complexity of Uniform Sampling]%{lemma}{lemuniform}
	\label{lemma:uniform}
	Given \eqref{eq:Hessian_boundedness_finite_sum_fi}, \eqref{eq:Hessian_boundedness_finite_sum_K} , and $0 < \epsilon,\delta < 1$, 
	let
	\begin{align}
	\label{eq:uniform_sample_size}
	|\mathcal S| \ge \frac{16 K_{\max}^2}{\epsilon^2}\log\frac{2 d}{\delta},
	\end{align}
	where $ K_{\max} $ is defined as in \eqref{eq:Hessian_boundedness_finite_sum_K}. At any $\x \in \mathbb{R}^{d}$, suppose picking the elements of $ \mathcal{S} $ uniformly at random with or without replacement, and forming $\H(\x)$ as in \eqref{eq:subsampled_H} with $ p_{i} = 1/n,; \forall i $. We have
	\begin{align}
	\label{eq:uniform_prob}
	\Pr\Big( \|\H(\x) - \nabla^{2} F(\x)\| \leq \epsilon \Big) \geq 1-\delta.
	\end{align}
\end{lemma}
%-------------------------------Lemma: uniform
\begin{proof}
	Consider $|\mathcal{S}|$ random matrices $\H_{j}(\x), j=1,\ldots,|\mathcal{S}|$ s.t.\ $\Pr \left(\H_{j}(\x) = \nabla^{2} f_{i}(\x)\right) = {1}/{n}; \; \forall i = 1,2,\ldots,n$. Define $\X_{j} \triangleq  \big( \H_{j}  - \nabla^{2}F(\x)\big) $, $\H \triangleq   \sum_{j \in \mathcal{S}} \H_{j}/|\mathcal{S}|$, and $\X \triangleq \sum_{j \in \mathcal{S}} \X_{j} = |\mathcal{S}| \left( \H - \nabla^{2}F(\x) \right)$. Note that $\Ex(\X_{j}) = 0$ and for $\H_{j} = \nabla^{2} f_{1}(\x)$ we have
	\begin{equation*}
	\|\X_{j}\|^{2} = \| \frac{n-1}{n} \nabla^{2} f_{1}(\x) - \sum_{i=2}^{n} \frac{1}{n} \nabla^{2} f_{i}(\x) \|^{2}\leq 4 (\frac{n-1}{n})^{2} K_{\max}^{2} \leq 4 K_{\max}^{2}.
	\end{equation*}
	Hence, we can apply Operator-Bernstein inequality \cite[Theorem 1]{gross2010note} to get
	\begin{align*}
	\Pr\Big(\|\H - \nabla^{2}F(\x) \| \geq \epsilon \Big) &= \Pr\Big(\|\X\| \geq \epsilon |\mathcal{S}| \Big)\leq 2d \exp\{-\epsilon^{2} |\mathcal{S}| /(16 K_{\max}^{2})\}.
	\end{align*}
	Now \eqref{eq:uniform_sample_size} ensure that $2 d \exp\{-\epsilon^{2} |\mathcal{S}| /(16 K_{\max}^{2})\} \leq \delta$,
	which gives \eqref{eq:uniform_prob}.
\qed
\end{proof}

Indeed, if \eqref{eq:uniform_prob} holds, then \eqref{eq:Hessian_accuracy_H} follows with the same probability. In addition, if $ H $ is constructed according to Lemma \ref{lemma:uniform}, it is easy to see that \eqref{eq:Hessian_boundedness_H} is satisfied with $ K_{H} = K_{\max} $ (in fact this is a deterministic statement). These two, together, imply that $ H $ satisfies Condition \ref{condition:Hessian_approximation_H}, with probability $ 1-\delta $.

\paragraph{A Special Case:} In certain settings, one might be able to construct a more ``informative'' distribution, $ \p $, over the indices in the set $\{1,2,\ldots,n\}$, as opposed to oblivious uniform sampling. In particular, it might be advantageous to bias the probability distribution towards picking indices corresponding to those $ f_{i} $'s which are more \emph{relevant}, in certain sense, in forming the Hessian. If this is possible, then we can only expect to require smaller sample size as compared with oblivious uniform sampling. One such setting where this is possible is the finite-sum optimization of the form \eqref{eq:obj_sum_ERM}, 
%\begin{align}
%%\label{eq:obj_sum_ERM}
%\min_{\x\in\bbR^d} F(\x)  \triangleq \frac{1}{n} \sum_{i=1}^n f_i(\a_i^T\x), \tag{{\bf P2}}
%\end{align}
%for some given data vectors $ \{\a_{i}\}_{i=1}^{n} \subset \mathbb{R}^{d} $. 
%Finite-sum problems of the form \eqref{eq:obj_sum_ERM}, 
which is indeed a special case of \eqref{eq:obj_sum} and arise often in many machine learning problems \cite{shalev2014understanding}.%, e.g., non-linear least squares arising from logistic regression with least squares loss as in \cite[Example 4.2]{xuNonconvexEmpirical2017}. 

It is easy to see that, the Hessian of $ F $ in this case can be written as
$\nabla^2 F(\x) = \A^T \B \A = \sum_{i=1}^n f_i''(\a_i^T\x)\a_i\a_i^T/n$, where
\begin{align*}
\A^{T} = \begin{bmatrix}
\mid & \mid & \dots & \mid \\
\a_{1} & \a_{2} & \dots & \a_{n}\\
\mid & \mid & \dots & \mid \\
\end{bmatrix}_{d \times n} \text{ and } \;
\B = \frac{1}{n}\begin{bmatrix}
f_1''(\a_1^T\x) & 0 & \dots  & 0 \\
0 & f_2''(\a_2^T\x) & \dots  & 0 \\
\vdots & \vdots & \ddots & \vdots \\
0 & 0 & \dots & f_n''(\a_n^T\x)
\end{bmatrix}_{n \times n}.
%B = \diag\{f_1''(\a_1^T\x), f_2''(\a_2^T\x)\cdots f_n''(\a_n^T\x)\}.
\end{align*}
Now let $ \S \in \mathbb{R}^{n \times |\mathcal{S}|} $ be the sampling matrix and define the approximate Hessian as $\H \triangleq \A^T\S\S^T\B\A$. It can be seen that approximating the Hessian matrix $\nabla^2 F(\x) = \A^T \B \A$ can be regarded as approximating matrix-matrix multiplication from RandNLA \cite{mahoney2011randomized,woodruff2014sketching}.
For this, consider the sampling distribution $ \p $ as 
\begin{align}
\label{eq:nonuniform_sampling_distribution}
p_i = \frac{|f_i''(\a_i^T\x)|\|\a_i\|_2^2}{\sum_{j=1}^n |f_j''(\a_j^T\x)|\|\a_j\|_2^2}.
\end{align}
Note that the absolute values are needed since for non-convex $f_i$, we might have $ f_j''(\a_j^T\x) < 0 $ (for the convex case where all $ f_j''(\a_j^T\x) \geq 0 $, one can obtain stronger guarantees than Lemmas \ref{lemma:uniform} and \ref{lemma:nonuniform}; see \cite{xu2016sub}). Using non-uniform sampling distribution \eqref{eq:nonuniform_sampling_distribution}, Lemma \ref{lemma:nonuniform} gives sampling complexity for the approximate Hessian of \eqref{eq:obj_sum_ERM} to, probabilistically, satisfy \eqref{eq:Hessian_approximation}. 
%The proof of Lemma \ref{lemma:nonuniform} is in Appendix \ref{sec:sampling_proofs}.

%-------------------------------%-------------------------------
%-------------------------------Lemma: non-uniform
\begin{lemma}[Complexity of Non-Uniform Sampling]%{lemma}{lemmatrix}
	\label{lemma:nonuniform}
    Given \eqref{eq:Hessian_boundedness_finite_sum_fi}, \eqref{eq:Hessian_boundedness_finite_sum_hat_K} and $0 < \epsilon, \delta < 1$, let 
	\begin{align}
	\label{eq:nonuniform_sample_size}
	|\mathcal S| \ge \frac{4 \widehat{K}^2}{\epsilon^2}\log\frac{2 d}{\delta},
	\end{align}
	where $\widehat K$ is defined as in \eqref{eq:Hessian_boundedness_finite_sum_hat_K}. 
	At any $\x \in \mathbb{R}^{d}$, suppose picking the elements of $ \mathcal{S} $ randomly according to the probability distribution \eqref{eq:nonuniform_sampling_distribution}, and forming $\H(\x)$ as in \eqref{eq:subsampled_H}. We have
	\begin{align}
	\label{eq:nonuniform_prob}
	\Pr\Big( \|\H - \nabla^{2} F(\x)\| \leq \epsilon \Big) \geq 1-\delta.
	\end{align}
\end{lemma}
%-------------------------------Lemma: non-uniform
\begin{proof}
Define $\B = \diag\{f_1''(\a_1^T\x)/n, \cdots, f_n''(\a_n^T\x)/n\}\in\bbR^{n\times n}$. Let $ \S \in \mathbb{R}^{n \times |\mathcal{S}|} $ be the sampling matrix and define $\H \triangleq \A^T\S\S^T\B\A$. Further, let the diagonals of $\B$ be denoted by $b_{i}$ and define
%\begin{equation*}
$c  \triangleq \sum_{i = 1}^{n} |b_{j}| \|\a_{j}\|^{2}$.
%\end{equation*}
Consider $s$ random matrices $\H_{j}$ such that
$\Pr( \H_{j} = b_{i} \a_{i} \a_{i}^{T}/p_{i} ) = p_{i}, \;\forall j = 1,2,\ldots,|\mathcal{S}|$, where $p_{i}  = {|b_{i}| \|\a_{i}\|^{2}}/({\sum_{i=1}^{n} |b_{j}| \|\a_{j}\|^{2}}).$
Define
\begin{align*}
\X_{j} \triangleq \H_{j} - \A^{T} \B \A, \quad \H \triangleq \frac{1}{|\mathcal{S}|} \sum_{j=1}^{|\mathcal{S}|} \H_{j}, \quad \X \triangleq \sum_{j=1}^{|\mathcal{S}|} \X_{j} = |\mathcal{S}| \left( \H - \A^{T} \B \A\right).
\end{align*}
Note that $\bbE[\X_{j}] = \sum_{i=1}^{n} p_{i}\left(b_{i} \a_{i} \a_{i}^{T}/p_{i}  - \A^{T} \B \A \right) = 0$, and
\begin{align*}
\bbE[\X_{j}^{2}]&= \bbE [ \H_{j} - \A^{T} \B \A] = \bbE [\H_{j}^{2}] + (\A^{T} \B \A)^{2} - \bbE[\H_{j}] \A^{T} \B \A - \A^{T} \B \A \bbE[\H_{j}] \\
&= \bbE[\H_{j}^{2}] - (\A^{T} \B \A)^{2} \preceq \bbE [\H_{j}^{2}]  = \sum_{i=1}^{n} p_{i} \left(\frac{b_{i}}{p_{i}} \a_{i} \a_{i}^{T} \right)^{2} = \sum_{i=1}^{n} \frac{b^{2}_{i} \|\a_{i}\|^{2}}{p_{i}} \a_{i} \a_{i}^{T}  
\\&= \sum_{i=1}^{n} |b_{j}| \|\a_{j}\|^{2} \sum_{i=1}^{n} |b|_{i}  \a_{i} \a_{i}^{T} = c  \sum_{i=1}^{n} |b|_{i}  \a_{i} \a_{i}^{T} = c\A^{T} |\B|\A.
\end{align*}
So we have
%\begin{equation*}
$\|\bbE[\X_{j}^{2}]\| \leq c  \|\A^{T} |\B| \A\|$.
%\end{equation*}
Now we can apply the  Operator-Bernstein inequality \cite[Theorem 1]{gross2010note} to get
$$
\Pr\left( \|\H - \A^{T} \B \A\|_{2} \geq \epsilon \right) \leq \Pr\left( \|\X\|_{2} \geq \epsilon |\mathcal{S}| \right) \leq 2 d e^{\epsilon^{2} |\mathcal{S}| /(4 c  \|\A^{T} |\B| \A\|)}.
$$
Since 
$
c = \sum_{i=1}^n \Abs{b_{i}}\norm{\a_i}^2 = \frac{1}{n} \sum_{i=1}^n \Abs{f_i''}\norm{\a_i}^2 \le \frac{1}{n}\sum_{i=1}^n K_i = \widehat K $  
and
$$
\norm{\A^T\Abs{\B}\A} = \norm{\frac{1}{n}\sum_{i=1}^n \Abs{f_i''}\a_i\a_i^T} \le \frac{1}{n}\sum_{i=1}^n \norm{\Abs{f_i''}\a_i\a_i^T} \le \frac{1}{n}\sum_{i=1}^n K_i = \widehat{K},
$$
then we have
$$
\Prob{\|\H - \A^{T} \B \A\|_{2} \geq \epsilon} \le 2 d e^{\epsilon^2 |\mathcal{S}|/ (4\widehat{K}^2)},
$$
which gives the desired result. \qed
\end{proof}

The bound in \eqref{eq:nonuniform_sample_size} can be improved by replacing the dimension $ d $ with a smaller quantity, known as intrinsic dimension; see Appendix A. 
As it can be seen from \eqref{eq:Hessian_boundedness_finite_sum_K} and \eqref{eq:Hessian_boundedness_finite_sum_hat_K}, since $\widehat{K} \leq K_{\max}$, the sampling complexity given by Lemma \ref{lemma:nonuniform} always provides a smaller sample-size compared with that prescribed by Lemma \ref{lemma:uniform}. Indeed, the advantage of non-uniform sampling is more pronounced in cases where the distribution of $ K_{i} $'s are highly skewed, i.e., a few large ones and many small ones, in which case we can have $\widehat{K} \ll K_{\max}$; see numerical experiments in \cite{xuNonconvexEmpirical2017}.
Also, from \eqref{eq:nonuniform_prob}, it follows that the approximate matrix $ \H $, constructed according to Lemma \ref{lemma:nonuniform} satisfies \eqref{eq:Hessian_boundedness_H} with $ K_{H} = \widehat{K} + \epsilon$, with probability $ 1-\delta $, which in turn, implies that Condition \ref{condition:Hessian_approximation_H} is ensured, with probability $ 1-\delta $.

As concrete examples of the problems in the form \eqref{eq:obj_sum_ERM} where Lemma \ref{lemma:nonuniform} can be readily used, Table \ref{table_finite_sum_example} gives estimates for $ K_{i} $ in \eqref{eq:Hessian_boundedness_finite_sum_fi} for robust linear regression with smooth non-convex bi-weight loss, \cite{beaton1974fitting}, as well as non-convex binary-classification using logistic regression with least squares loss, \cite{xuNonconvexEmpirical2017}.
\begin{table}[htbp]
	\caption{Examples of problems in the form \eqref{eq:obj_sum_ERM} with the corresponding estimates for $ K_{i} $ in \eqref{eq:Hessian_boundedness_finite_sum_fi}. \label{table_finite_sum_example}}
	\centering
	\scalebox{0.9}{
		\begin{tabular}{| c | c | c | c | c |} 
			%\hline
			%\multicolumn{4}{|c|}{$G$ for GLMs with sparsity constraint} \\
			\hline& & & & \\ [-2ex]
			Problem & \parbox[t]{1cm}{\centering Data}  & $ \displaystyle f_i(\a_i^T\x) $ & $\displaystyle \nabla^{2} f_{i}(\a_i^T\x)$ & $\displaystyle  K_{i}$\\ [0.5ex] 
			\hline & & & &\\ [-1ex]
			\parbox[t]{2.1cm}{\centering Robust Linear \\ Regression}   & \parbox[t]{1cm}{\centering $\displaystyle \a_{i} \in \mathbb{R}^{d}$ \\ $b_{i} \in \mathbb{R}$} & $\displaystyle  \frac{\left(\a_i^T\x - b_{i}\right)^{2}}{1+\left(\a_i^T\x - b_{i}\right)^{2}}$ & $\displaystyle \left(\frac{2\left(1-3 \left(\a_i^T\x\right)^2\right)}{\left(\left(\a_i^T\x\right)^2 + 1\right)^3} \right) \a_{i} \a_{i}^{T}$ & $\displaystyle  \frac{\|\a_{i}\|^{2}}{6 \sqrt{3}}$\\ [5ex]
			\parbox[t]{2.1cm}{\centering Non-linear Binary  \\ Classification} & \parbox[t]{1.2cm}{\centering $\displaystyle \a_{i} \in \mathbb{R}^{d}$ \\ $b_{i} \in \left\{0,1\right\}$} & $\displaystyle \left(\frac{1}{1+\exp\left(-\a_i^T\x\right)} - b_{i}\right)^{2}$ & $\displaystyle \left(\frac{\exp\left(\a_i^T\x\right)\left(1-\exp\left(\a_i^T\x\right)\right)}{\left(\exp\left(\a_i^T\x\right) + 1\right)^3} \right) \a_{i} \a_{i}^{T}$ & $\displaystyle  2 \|\a_{i}\|^{2}$\\ [3ex]
			\hline
	\end{tabular}}

\end{table}

%\begin{table}[htbp]
%	\caption{Examples of problems in the form \eqref{eq:obj_sum_ERM} with the corresponding estimates for $ K_{i} $ in \eqref{eq:Hessian_boundedness_finite_sum_fi}. \label{table_finite_sum_example}}
%	\centering
%	\scalebox{1}{
%		\begin{tabular}{|c | c | c |} 
%			%\hline
%			%\multicolumn{4}{|c|}{$G$ for GLMs with sparsity constraint} \\
%			\hline & & \\ [-2ex]
%			& Robust Linear Regression & Non-linear Binary Classification\\ [0.5ex] 
%			\hline & & \\ [-1ex]
%			Data  & $\displaystyle \left\{ \a_{i},b_{i} \right\}_{i=1}^{n} \subset \mathbb{R}^{d} \times \mathbb{R}$ & $\displaystyle \left\{ \a_{i},b_{i} \right\}_{i=1}^{n} \subset \mathbb{R}^{d} \times \left\{0,1\right\}$ \\ [2ex]
%			$\displaystyle f_i(\a_i^T\x)$ & $\displaystyle  \frac{\left(\a_i^T\x - b_{i}\right)^{2}}{1+\left(\a_i^T\x - b_{i}\right)^{2}}$ & $\displaystyle \left(\frac{1}{1+\exp\left(-\a_i^T\x\right)} - b_{i}\right)^{2}$ \\ [3ex]
%			$\displaystyle \nabla^{2} f_{i}(\a_i^T\x)$ & $\displaystyle \left(\frac{2\left(1-3 \left(\a_i^T\x\right)^2\right)}{\left(\left(\a_i^T\x\right)^2 + 1\right)^3} \right) \a_{i} \a_{i}^{T}$ & $\displaystyle \left(\frac{\exp\left(\a_i^T\x\right)\left(1-\exp\left(\a_i^T\x\right)\right)}{\left(\exp\left(\a_i^T\x\right) + 1\right)^3} \right) \a_{i} \a_{i}^{T}$  \\ [4ex]
%			$\displaystyle  K_{i}$ & $\displaystyle  \frac{\|\a_{i}\|^{2}}{6 \sqrt{3}}$ & $\displaystyle  2 \|\a_{i}\|^{2}$  \\ [2ex] 
%			%Reference & \cite{beaton1974fitting} & \cite{xuNonconvexEmpirical2017}  \\ [1ex] 
%			\hline
%	\end{tabular}}
%	
%\end{table}

\subsection{Probabilistic Convergence Analysis}
\label{sec:probabilistic_convergence}
Now, we are in the position to give iteration complexity for Algorithms \ref{alg:STR_fg} and \ref{alg:SARC_fg} where the inexact Hessian matrix $ \H_{t} $ is constructed according to Lemmas \ref{lemma:uniform} or \ref{lemma:nonuniform}. Since the approximation is a probabilistic construction, in order to guarantee success, we need to ensure that we require a small failure probability across all iterations. In particular, in order to get an overall and accumulative success probability of $ 1 - \delta $ for the entire $ T  $ iterations, the per-iteration failure probability is set as $(1- \sqrt[T]{(1-\delta)} )\in \mathcal{O}(\delta/T)$. This failure probability appears only in the ``log factor'' for sample size in all of our results, and so it is not the dominating cost. Hence, requiring that all $ T $ iterations are successful for a large $ T $, only necessitates a small (logarithmic) increase in the sample size. For example, for $ T \in \bigO (\max\{\epsilon_g^{-2}, \epsilon_H^{-3}\})$, as in Theorem \ref{theorem:SARC_main_det}, we can set the per-iteration failure probability to $ \delta \min\{\epsilon_g^{2}, \epsilon_H^{3}\} $, and ensure that when Algorithm \ref{alg:SARC_fg} terminates, all Hessian approximations have been, accumulatively, successful with probability of $ 1-\delta $. 
%As a result, we can use the sub-sampling Lemmas \ref{lemma:uniform} and \ref{lemma:nonuniform} and obtain an approximating matrix which, with high probability, guarantees \eqref{eq:Hessian_approximation}. 

Using these results, we can have the following probabilistic, but optimal, guarantee on the worst-case iteration complexity of Algorithm \ref{alg:STR_fg} for solving finite-sum problem \eqref{eq:obj_sum} (or \eqref{eq:obj_sum_ERM}) and in the case where the inexact Hessian is formed by sub-sampling. Their proofs follow very similar line of reasoning as that used for obtaining the results of Section \ref{sec:convergence_analysis}, and hence are omitted.

%-------------------------------%-------------------------------
\begin{theorem}[Optimal Complexity of Algorithm \ref{alg:STR_fg} For Finite-Sum Problem]%{theorem}{theoremSTR_prob}
	\label{theorem:STR_main_prob}
	Consider any $ 0< \epsilon_{g},\epsilon_{H}, \delta < 1 $. Let $ \epsilon $ be as in \eqref{eq:STR_epsilon} and set $ \delta_{0} = \delta \min\{\epsilon_g^{2} \epsilon_H, \epsilon_H^{3}\}$. 
	Furthermore, for such $ (\epsilon,\delta_{0}) $, let the sample-size $ |\mathcal{S}| $ be as in \eqref{eq:uniform_sample_size} (or \eqref{eq:nonuniform_sample_size}) and form the sub-sampled matrix $ \H $ as in \eqref{eq:subsampled_H}. 
	For Problem \eqref{eq:obj_sum} (or \eqref{eq:obj_sum_ERM}), under Assumption \ref{assumption:a1} and Condition \ref{condition:STC_sufficient_descent}, Algorithm \ref{alg:STR_fg} terminates in  at most $T \in \bigO(\max\{\epsilon_g^{-2} \epsilon_H^{-1}, \epsilon_H^{-3}\})$
%	\end{align*}
	iterations, upon which, with probability $ 1-\delta $, we have that $\|\nabla F(\x)\| \le \epsilon_g$, and $\lambda_{\min} (\nabla^2 F(\x)) \ge -\left(\epsilon + \epsilon_H\right)$.
%	\begin{align*}
%	\|\nabla F(\x)\| \le \epsilon_g, \quad \text{and} \quad \lambda_{\min} (\nabla^2 F(\x)) \ge -\left(\epsilon + \epsilon_H\right).
%	\end{align*}
\end{theorem}
%-------------------------------%-------------------------------
%-------------------------------%-------------------------------
Similarly, in the setting of optimization problems \eqref{eq:obj_sum} and \eqref{eq:obj_sum_ERM}, with appropriate sub-sampling of the Hessian as in Lemmas \ref{lemma:uniform} and \ref{lemma:nonuniform}, we can also obtain probabilistic worst-case iteration complexities for Algorithm \ref{alg:SARC_fg} as in the deterministic case. Again, the proofs are similar to those in Section \ref{sec:convergence_analysis}, and hence are omitted. 

\begin{theorem}[Complexity of Algorithm \ref{alg:SARC_fg} For Finite-Sum Problem]%{theorem}{theoremSARC_prob}
	\label{theorem:SARC_main_prob}
	Consider any $ 0< \epsilon_{g},\epsilon_{H},\delta < 1 $. Let $ \epsilon $ be as in \eqref{eq:SARC_epsilon} and set $ \delta_{0} = \delta \min\{\epsilon_g^{2}, \epsilon_H^{3}\}$. Furthermore, for such $ (\epsilon,\delta_{0}) $, let the sample-size $ |\mathcal{S}| $ be as in \eqref{eq:uniform_sample_size} (or \eqref{eq:nonuniform_sample_size}) and form the sub-sampled matrix $ \H $ as in \eqref{eq:subsampled_H}. 
	For Problem \eqref{eq:obj_sum} (or \eqref{eq:obj_sum_ERM}), under Assumption \ref{assumption:a1} and Condition \ref{condition:SARC_sufficient_descent}, Algorithm \ref{alg:SARC_fg} terminates in at most  $T \in \bigO(\max\{\epsilon_g^{-2}, \epsilon_H^{-3}\})$	iterations, upon which, with probability $ 1-\delta $, we have that $\|\nabla F(\x)\| \le \epsilon_g$, and $\lambda_{\min} (\nabla^2 F(\x)) \ge -\left(\epsilon + \epsilon_H\right)$.
%	\begin{align*}
%	\|\nabla F(\x)\| \le \epsilon_g, \quad \text{and} \quad \lambda_{\min} (\nabla^2 F(\x)) \ge -\left(\epsilon + \epsilon_H\right).
%	\end{align*}
\end{theorem}
%-------------------------------%-------------------------------
%-------------------------------%-------------------------------
\begin{theorem}[Optimal Complexity of Algorithm \ref{alg:SARC_fg} For Finite-Sum Problem]%{theorem}{theoremSARCOptimal_prob}
	\label{theorem:SARC_main_prob_optimal}
	Consider any $ 0< \epsilon_{g},\epsilon_{H},\delta < 1 $. Let $ \epsilon $ be as in Theorem \ref{theorem:SARC_main_det_optimal} and set $ \delta_{0} = \delta \min\{\epsilon_g^{3/2}, \epsilon_H^{3}\}$. 
	Furthermore, for such $ (\epsilon,\delta_{0}) $, let the sample-size $ |\mathcal{S}| $ be as in \eqref{eq:uniform_sample_size} (or \eqref{eq:nonuniform_sample_size}) and form the sub-sampled matrix $ \H $ as in \eqref{eq:subsampled_H}. 
	For Problem \eqref{eq:obj_sum} (or \eqref{eq:obj_sum_ERM}), under Assumption \ref{assumption:a1}, Conditions \ref{condition:SARC_sufficient_descent} and \ref{condition:SARC_sufficient_descent_strict}, Algorithm \ref{alg:SARC_fg} terminates in at most  $T \in \bigO (\max\{\epsilon_g^{-3/2}, \epsilon_H^{-3}\})$ iterations, upon which, with probability $ 1-\delta $, we have that $\|\nabla F(\x)\| \le \epsilon_g$, and $\lambda_{\min} (\nabla^2 F(\x)) \ge -\left(\epsilon + \epsilon_H\right)$.
%	\begin{align*}
%	\|\nabla F(\x)\| \le \epsilon_g, \quad \text{and} \quad \lambda_{\min} (\nabla^2 F(\x)) \ge -\left(\epsilon + \epsilon_H\right).
%	\end{align*}
\end{theorem}
%-------------------------------%-------------------------------
As it can be seen, the main difference between Theorems \ref{theorem:SARC_main_prob} and \ref{theorem:SARC_main_prob_optimal} is in the solution to the sub-problem \eqref{eq:SARC_subp}. More specifically, if in addition to Condition \ref{condition:SARC_sufficient_descent}, Condition \ref{condition:SARC_sufficient_descent_strict} is also satisfied, then Theorem \ref{theorem:SARC_main_prob_optimal} gives \emph{optimal} worst-case iteration complexity.
\section{Conclusion}
\label{sec:conclusions}
We considered non-convex optimization settings and developed efficient variants of the trust region and adaptive cubic regularization methods in which both the sub-problems as well as the the curvature information are suitably approximated. For all of our proposed variants, we obtained iteration complexities to achieve approximate second order criticality, which are shown to be the same (up to some constant) as that of the exact variants.
	
As compared with previous works, our proposed Hessian approximation condition offers a range of theoretical and practical advantages. As a concrete example, we considered the large-scale finite-sum optimization problem and proposed uniform and non-uniform sub-sampling strategies as ways to efficiently construct the desired approximate Hessian. We then, probabilistically, established optimal iteration complexity for variants of trust region and adaptive cubic regularization methods in which the Hessian is appropriately sub-sampled. 

In this paper, we focused on approximating the Hessian under the exact gradient information. Arguably, the bottleneck of the computations in such second-order methods involves the computations with the Hessian, e.g., matrix-vector products in the (approximate) solution of the sub-problem. In fact, the cost of the exact gradient computation is typically amortized by that of the operations with the Hessian. In spite of this, approximating the gradient in a computationally feasible way and with minimum assumptions could improve upon the efficiency of the methods proposed here. However, care has to be taken as cheaper iterations with inaccurate gradients could in fact result in more iterations overall. This could have the adverse effect of slowing down the algorithm's convergence. As a result, approximating the gradient has to be done with care to avoid such pitfalls.

Finally, we mention that our focus here has been solely on developing the theoretical foundations of such randomized algorithms. Extensive empirical evaluations of these algorithms on various machine learning applications are given in the \cite{xuNonconvexEmpirical2017}.

%\subsubsection*{Acknowledgment}
%We would like to acknowledge ARO, DARPA, and NSF for providing partial support of this work. 
\printbibliography
\section*{Appendix A: Intrinsic dimension and improving the sampling complexity \eqref{eq:nonuniform_sample_size}}
We can still improve the sampling complexity \eqref{eq:nonuniform_sample_size} by considering the intrinsic dimension of the matrix $\A^{T} |\B| \A$. Recall that for a SPSD matrix $ \A \in \mathbb{R}^{d \times d}$, the intrinsic dimension is defined as $ t(A) =  \text{tr}(\A)/\|\A\|$, where $ \text{tr}(\A) $ is the trace of $ \A $. The intrinsic dimension can be regarded as a measure for the number of dimensions where $\A$ has a significant spectrum. It is easy to see that $ 1 \leq t(\A) \leq \text{rank}(\A) \leq d $; see \cite{tropp2015introduction} for more details. Now let $t = \text{tr}(\A^{T} |\B| \A)/\|\A^{T} |\B| \A\|$ be the intrinsic dimension of the SPSD matrix $ \A^{T} |\B| \A $. 
We have the following improved sampling complexity result:

\begin{lemma}[Complexity of Non-Uniform Sampling: Intrinsic Dimension]%{lemma}{lemmatrix}
	\label{lemma:nonuniform_intrinsic}
	The result of Lemma \ref{lemma:nonuniform} holds with \eqref{eq:nonuniform_sample_size} replaced with
	\begin{align}
		\label{eq:nonuniform_sample_size_intrinsic}
		|\mathcal S| \ge \frac{16 \widehat{K}^2}{3\epsilon^2}\log\frac{8t}{\delta},
	\end{align}
	where $t = \text{tr}(\A^{T} |\B| \A)/\|\A^{T} |\B| \A\| \leq d$ is the intrinsic dimension of the matrix $ \A^{T} |\B| \A $.
\end{lemma}	

\begin{proof}
	It is easy to see that $\text{Var}(\X) = \bbE(\X^{2})  = \sum_{j=1}^{|\mathcal{S}|} \bbE(\X_{j}^{2}) \preceq |\mathcal{S}| c \A^{T} |\B| \A$, where $ \X $ and $c$ are given in the proof of Lemma \ref{lemma:nonuniform}. For $\H_{j} = \frac{b_{i}}{p_{i}} \a_{i} \a_{i}^{T}$, we have
	\begin{align*}
		&\lambda_{\max}(\X_{j})\leq \|\X_{j}\| = \| \frac{b_{i}}{p_{i}} \a_{i} \a_{i}^{T} - \A^{T} \B \A \|  = \|\left(\frac{1-p_{i}}{p_{i}}\right) b_{i} \a_{i} \a_{i}^{T}  - \sum_{j \neq i} b_{j} \a_{j} \a_{j}^{T} \| \\
		&~ \leq \left(\frac{1-p_{i}}{p_{i}}\right) |b_{i}| \|\a_{i}\|^{2}  + \sum_{j \neq i} |b_{j}| \|\a_{j}\|^{2} = \left(1-p_{i}\right) \sum_{i=1}^{n} |b_{j}| \|\a_{j}\|^{2} + \sum_{j \neq i} |b_{j}| \|\a_{j}\|^{2} \\
		&~~= 2 \sum_{j \neq i} |b_{j}| \|\a_{j}\|^{2} \leq 2 \sum_{i = 1}^{n} |b_{j}| \|\a_{j}\|^{2} = 2 c.
	\end{align*}
	Hence, if $\epsilon |\mathcal{S}| \geq \sqrt{|\mathcal{S}| c \|\A^{T} |\B| \A\|}  + 2 c/3$, we can apply Matrix Bernstein using the intrinsic dimension \cite[Theorem 7.7.1]{tropp2015introduction} to
	get for $\epsilon \leq 1/2$
	\begin{align*}
		\Pr\left( \lambda_{\max}(\X) \geq \epsilon |\mathcal{S}| \right) &\leq 4 t \exp\left\{\frac{-\epsilon^2 |\mathcal{S}|}{2 c \|\A^{T} |\B| \A\| + 4 c \epsilon/3}\right\}\leq 4 t \exp\left\{\frac{-3 \epsilon^2 |\mathcal{S}|}{16 c^{2}}\right\}.
	\end{align*}
	Applying the same bound for $\Y_{j} = - \X_{j}$ and $\Y = \sum_{j=1}^{s} \Y_{j}$, followed by the union bound, we get the desired result.
	\qed
\end{proof}

\section*{Appendix B: Computation of Approximate Negative Curvature Direction}
Throughout our analysis, we assume that, if a sufficiently negative curvature exists, i.e., $ \lambda_{\min}(\H) \leq -\epsilon_{H}$ for some $ \epsilon_{H} \in (0,1) $, we can approximately compute the corresponding negative curvature direction vector $ \u $, i.e., $\lin{\u, \H \u} \leq -  \nu \epsilon_{H} \|\u\|^{2}$, for some $ \nu \in (0,1) $. We note that this can be done efficiently by applying a variety of methods such as Lanczos \cite{kuczynski1992estimating} or shift-and-invert \cite{garber2016faster} on the SPSD matrix $ \tilde{\H} = K_{H} - \H $. These methods only employ matrix vector products and, hence, are suitable for large scale problems. More specifically, with any $ \kappa \in (0,1) $, these methods using $ \mathcal{O}(\log(d/\delta)\sqrt{K_{H}/\kappa}) $ matrix-vector products and with probability $ 1-\delta $, yield a vector $ \u $ satisfying
$K_{H}\|\u\|^{2} - \lin{\u, \H \u}  = \lin{\u, \tilde{\H} \u} \geq \kappa \lambda_{\min}(\tilde{\H}) \|\u\|^{2} = \kappa (K_{H} - \lambda_{\min}(\H)) \|\u\|^{2}$. Rearranging, we obtain
$\lin{\u, \H \u} \leq (1-\kappa) K_{H} \|\u\|^{2} + \kappa \lambda_{\min}(\H) \|\u\|^{2}$. Setting 
$1 > \nu = 2\kappa \geq (2 K_{H})/(2 K_{H} + \epsilon_{H})$, gives $\lin{\u, \H \u} \leq -\nu \epsilon_{H} \|\u\|^{2}$. %See \cite{carmon2016accelerated} for another application of this technique in obtaining approximate negative curvature direction for non-convex optimization. 

% -----------------------------------------------------------------------
%%%%%%%%%%%%%%%%%%%%%%%%%%%%%% APPENDIX %%%%%%%%%%%%%%%%%%%%%%%%%%%%%%%%%
% -----------------------------------------------------------------------
\end{document}